\documentclass{article}
\usepackage[affil-it]{authblk}
\usepackage{amssymb,amsmath,amsthm}
\usepackage{graphicx}
\usepackage{hyperref}
\usepackage[margin=1in]{geometry}
\usepackage{stackrel}
\usepackage{tensor}
\newcommand{\ie}{\emph{i.e.}}
\newcommand{\eg}{\emph{e.g.}}
\newcommand{\cf}{\emph{cf.}}
\newcommand{\Nat}{\mathbb{N}}
\newcommand{\Real}{\mathbb{R}}

\newcommand{\Int}{\mathbb{Z}}
\newcommand{\sii}{L^2}
\newcommand{\Dom}{\mathsf{D}}

\newcommand{\rot}{\mathop{\mathrm{rot}}\nolimits}
\newcommand{\dist}{\mathop{\mathrm{dist}}\nolimits}
\newcommand{\divergence}{\mathop{\mathrm{div}}\nolimits}
\newcommand{\supp}{\mathop{\mathrm{supp}}\nolimits}
\newcommand{\eps}{\varepsilon}

\newcommand{\Euler}{\mathrm{e}}
\newcommand{\der}{\mathrm{d}}
\newcommand{\chart}{\pi}
\newtheorem{Lemma}{Lemma}[section]
\newtheorem{Theorem}{Theorem}[section]
\newtheorem{Proposition}{Proposition}[section]
\newtheorem{Corollary}{Corollary}[section]
\theoremstyle{definition}
\newtheorem{Remark}{Remark}[section]
\numberwithin{equation}{section}
\usepackage[normalem]{ulem}
\usepackage{color}
\definecolor{DarkGreen}{rgb}{0,0.5,0.1}
\definecolor{DarkBlue}{rgb}{0,0.1,0.7}
\definecolor{LightRed}{rgb}{1.00,0.04,0.13}

\newcommand\soutD{\bgroup\markoverwith
{\textcolor{DarkBlue}{\rule[.5ex]{2pt}{1pt}}}\ULon}
\newcommand{\Hm}[1]{\leavevmode{\marginpar{\tiny%
$\hbox to 0mm{\hspace*{-0.5mm}$\leftarrow$\hss}%
\vcenter{\vrule depth 0.1mm height 0.1mm width \the\marginparwidth}%
\hbox to
0mm{\hss$\rightarrow$\hspace*{-0.5mm}}$\\\relax\raggedright #1}}}

\newcommand\soutC{\bgroup\markoverwith
{\textcolor{LightRed}{\rule[.5ex]{2pt}{1pt}}}\ULon}

\begin{document}
%
\title{\textbf{\Large
The Hardy inequality and the heat equation
with magnetic field in any dimension}}


\author{Cristian Cazacu$^{a}$ \ and \ David Krej\v{c}i\v{r}\'ik$^{b}$}

\date{
\vspace{-5ex}
\small
\emph{
\begin{quote}
\begin{itemize}
\item[$a)$] 
Department of Mathematics and Informatics, Faculty of Applied Sciences, University Poli\-teh\-ni\-ca of Bucharest,
Splaiul Independentei 313, 060042 Bucharest, 
\& 
Simion Stoilow Institute of Mathematics of the Romanian Academy, Research Group of the Project PN-II-ID-PCE-2011-3-0075,
 21 Calea Grivitei Street, 010702 Bucharest, Romania;
\texttt{\emph{cristi\_cazacu2002@yahoo.com}}
\\
\item[$b)$]
Department of Theoretical Physics, Nuclear Physics Institute ASCR,
25068 \v{R}e\v{z}, Czech Republic;
\texttt{\emph{krejcirik@ujf.cas.cz}}%
\end{itemize}
\end{quote}
}
\vspace{-3ex}}

\maketitle

\begin{abstract}
\noindent
In the Euclidean space of any dimension~$d$,
we consider the heat semigroup generated by the magnetic
Schr\"odinger operator from which an inverse-square potential
is subtracted in order to make the operator critical
in the magnetic-free case.
Assuming that the magnetic field is compactly supported,
we show that the polynomial large-time behaviour of the heat semigroup
is determined by the eigenvalue problem
for a magnetic Schr\"odinger operator
on the $(d-1)$-dimensional sphere whose vector potential
reflects the behaviour of the magnetic field at the space infinity.
From the spectral problem on the sphere,
we deduce that in $d=2$ there is an improvement of the decay rate of
the heat semigroup by a polynomial factor
with power proportional to the distance
of the total magnetic flux to the discrete set of flux quanta,
while there is no extra polynomial decay rate in higher dimensions.
To prove the results, we establish new magnetic Hardy-type inequalities
for the Schr\"odinger operator and develop the method of self-similar variables
and weighted Sobolev spaces for the associated heat equation.
%
%
\end{abstract}
%


%
\section{Introduction}
%
This paper is concerned with the large-time behaviour of
the heat semigroup
\begin{equation}\label{semigroup}
  \Euler^{-t H_B}
\end{equation}
generated by the magnetic Schr\"odinger operator
\begin{equation}\label{Schrodinger}
  H_{B} = \big(-i\nabla_{\!x}-A(x)\big)^2 - \frac{c_d}{|x|^2}
  \qquad \mbox{in} \qquad
  \sii(\Real^d)
  \,.
\end{equation}
The relationship between
the \emph{magnetic potential} ($1$-form) $A:\Real^d \to \Real^d$
and the associated \emph{magnetic tensor} ($2$-form)~$B$ is standard,
through the exterior derivative
\begin{equation}\label{magnet.intro}
  B = \der A
  \,.
\end{equation}
The latter is compatible because of the (second) Maxwell equation
(Gauss' law for magnetism reflecting the absence of magnetic monopoles)
\begin{equation}\label{Gauss}
  \der B = 0
  \,,
\end{equation}
whose mathematical meaning is that~$B$ is a closed form.
The dimensional quantity~$c_d$ in~\eqref{Schrodinger}
is the best constant in
the \emph{classical Hardy inequality}
\begin{equation}\label{Hardy}
  \forall \psi \in C_0^\infty(\Real^d)
  \,, \quad
  \int_{\Real^d} |\nabla\psi(x)|^2 \, \der x
  \geq c_d
  \int_{\Real^d} \frac{|\psi(x)|^2}{|x|^2} \, \der x
  \,.
\end{equation}
It is well known that 
$$
  c_d = \left(\frac{d-2}{2}\right)^2
$$
if $d \geq 2$ (\eqref{Hardy}~holds as a trivial inequality if $d=2$).
We could conventionally put also $c_1=0$, 
but the one-dimensional situation will not be
considered in this paper, because there is no magnetic field in~$\Real$.

Clearly, $u(x,t):=\Euler^{-t H_B} u_0(x)$ is a solution of the Cauchy problem
\begin{equation}\label{Cauchy}
\left\{
\begin{aligned}
  \frac{\partial u}{\partial t} + H_B \;\! u &= 0
  \,, \qquad
  \\
  u(x,0) &= u_0(x)
  \,,
\end{aligned}
\right.
\end{equation}
where $(x,t) \in \Real^d\times(0,\infty)$ and $u_0 \in \sii(\Real^d)$.
Having the classical interpretation of the heat equation in mind,
it is thus possible to think of~$u$ as a temperature distribution
of a magnetic-sensitive medium in~$\Real^d$.
However, our main motivation to consider~\eqref{semigroup}
is its relevance in quantum mechanics,
despite the fact that the time evolution is given there
by the Schr\"odinger group, \cf~\cite{Simon_1982}.
In this context, $H_B$~is the Hamiltonian of a non-relativistic particle
interacting with the magnetic field~$B$
and a stationary electric dipole (see, \eg, \cite{Connolly-Griffiths_2007}).
We refer to the seminal paper \cite{Avron-Herbst-Simon_1978}
on a rigorous study of the magnetic field in quantum mechanics
and to~\cite{Erdos_2007} for a recent review with many references.
Finally, let us point out that~\eqref{Cauchy} has a stochastic
interpretation through the Brownian motion with imaginary drift,
\cf~\cite[Sec.~V]{Simon_Functional-new}.

Mathematically, we subtract the inverse-square potential in~\eqref{Schrodinger}
in order to reveal the \emph{transient effect} of the magnetic field.
It is well known that the large-time behaviour of a heat semigroup
is determined by spectral-threshold properties of its generator.
An important characterisation of this threshold behaviour
is given by the existence/non-existence of Hardy-type inequalities.
In the absence of magnetic field,
$H_0:= -\Delta_{x}-c_d/|x|^2$ is critical
in the sense that $c_d$ is optimal in \eqref{Hardy}
and no other non-trivial reminder term could be added
on the right hand side of~\eqref{Hardy}.
On the other hand, the following \emph{magnetic Hardy inequality}
holds whenever~$B$ is non-trivial
(in this case we write $B \not= 0$,
and similarly for other relations between functions).
\begin{Theorem}\label{Thm.Hardy}
Let $d \geq 2$.
Suppose that~$B$ is smooth and closed.
If $B\not=0$, then there exists a positive constant~$c_{d,B}$ such that
for any smooth~$A$ satisfying $\der A = B$,
the following inequality holds
\begin{equation}\label{Hardy.magnet}
  \forall \psi \in C_0^\infty(\Real^d) \,, \quad
  \int_{\Real^d}
  |(\nabla-iA)\psi(x)|^2 \, \der x
  - c_d
  \int_{\Real^d} \frac{|\psi(x)|^2}{|x|^2} \, \der x
  \geq c_{d,B}
  \int_{\Real^d} \frac{|\psi(x)|^2}{1+|x|^2 \log^2(|x|)} \, \der x
  \,.
\end{equation}
\end{Theorem}

This inequality was first proved by Laptev and Weidl in
\cite{Laptev-Weidl_1999} in $d=2$ under a flux condition
and with a better weight (without the logarithm)
on the right hand side of~\eqref{Hardy.magnet},
\cf~Theorem~\ref{Thm.Hardy.LW} below.
A general version of~\eqref{Hardy.magnet},
but with the integral on the right hand side being replaced by an integration
over a compact set of~$\Real^d$, was given by Weidl in \cite{Weidl_1999}.
We also refer to~\cite{Alziary-Fleckinger-Pelle-Takac_2003},
\cite{Balinsky-Laptev-Sobolev_2004}, \cite{Evans-Lewis_2005},
\cite[Sec.~6]{Kovarik_2011}
and~\cite{Ekholm-Portmann_2014} for related works.
In the last reference the authors establish
a variant of~\eqref{Hardy.magnet}
in $d=3$ under an extra assumption on~$B$.
Since the present version of the magnetic Hardy inequality
(in any dimension, with the minimal assumption $B\not=0$
and with an everywhere positive Hardy weight)
does not seem to exist in the literature,
we give a proof of Theorem~\ref{Thm.Hardy}
before proving the main result of this paper.
In the latter we essentially use the two-dimensional
variant of~\eqref{Hardy.magnet} due to Laptev and Weidl
that we therefore reprove in Theorem~\ref{Thm.Hardy.LW}.

Let us now come back to the transient effect of the magnetic field
as regards the large-time behaviour of~\eqref{semigroup}.
Assuming that~$A$ is smooth,
the \emph{diamagnetic inequality}
(see, \eg, \cite[Thm.~7.21]{LL} or \cite[Thm.~2.1.1]{Fournais-Helffer_2009})
\begin{equation}\label{diamagnetic}
  \big|(\nabla-iA)\psi(x)\big| \geq \big|\nabla|\psi|(x)\big|
\end{equation}
holds pointwise for almost every $x \in \Real^d$
and any $\psi \in H_\mathrm{loc}^1(\Real^d)$.
Consequently, $\inf\sigma(H_B) \geq 0 = \inf\sigma(H_0)$
and the spectral mapping theorem then yields
$
  \|\Euler^{-t H_B}\|
  \leq 1 = \|\Euler^{-t H_0}\|
$.
Hence, the decay of the heat semigroup in the presence
of magnetic field can be only better with respect to $B=0$.
This is notably evident for non-trivial homogeneous fields,
\ie\ $B(x)=B_0\not=0$ for all $x \in \Real^d$,
when the inequality is actually strict.
Indeed, $\lambda_1 := \inf\sigma(H_{B_0}) > 0$ in this case
(see~\cite{Balinsky-Laptev-Sobolev_2004}
for more general conditions on~$B$ to have
the positivity of the spectral threshold)
and we thus get an exponential decay
$
  \|\Euler^{-t H_B}\|
  \leq \Euler^{-t \lambda_1}
$.

In this paper we are interested in a more delicate situation
when~$B$ is \emph{local} in the sense that
it decays sufficiently fast at infinity so that
\begin{equation}\label{spectrum}
  \sigma(H_B) = \sigma(H_0) = [0,\infty) \,.
\end{equation}
Then
$
  \|\Euler^{-t H_B}\|
  = 1
$
and no extra decay of the heat semigroup is seen at this level. 
Although the spectrum as a set is insensitive
to this class of magnetic fields,
it follows from Theorem~\ref{Thm.Hardy}
that there is a fine difference reflected
in the presence of the magnetic Hardy inequality.
To exploit this subtle repulsive property of the magnetic field,
we introduce a weighted space
\begin{equation}\label{weighted}
  \sii_w(\Real^d) := \sii(\Real^d,w(x)\,\der x)
  \,, \qquad \mbox{where} \qquad
  w(x) := \Euler^{|x|^2/4}
  \,,
\end{equation}
and reconsider~\eqref{semigroup} as an operator
from $\sii_w(\Real^d) \subset \sii(\Real^d)$ to~$\sii(\Real^d)$.
That is, we restrict the initial data~$u_0$ in~\eqref{Cauchy} to lie in $\sii_w(\Real^d)$.
As a measure of the additional decay of the heat semigroup,
we then consider the polynomial \emph{decay rate}
\begin{equation}\label{rate}
  \gamma_B
  := \sup \Big\{ \gamma \ \Big| \
  \exists C_\gamma > 0, \, \forall t \geq 0, \
  \big\|\Euler^{-t H_B}\big\|_{
  \sii_w(\Real^d)
  \to
  \sii(\Real^d)
  }
  \leq C_\gamma \, (1+t)^{-\gamma}
  \Big\}
  \,.
\end{equation}
It is not difficult to see that $\gamma_0=1/2$ for any $d \geq 2$.	
The primary objective of this work is to study the influence
of a local but non-trivial magnetic field~$B$ on~$\gamma_B$.
Our main result reads as follows.
\begin{Theorem}\label{Thm.main.pre}
Let $d \geq 2$.
Suppose that~$B$ is smooth, closed and compactly supported.
Then
$$
  \gamma_B =
  \begin{cases}
  \displaystyle
  \frac{1+\beta}{2}
  & \mbox{if} \quad d=2 \,,
  \\
  \displaystyle
  \frac{1}{2}
  & \mbox{if} \quad d \geq 3 \,,
  \end{cases}
$$
where
\begin{equation}\label{flux}
  \beta := \dist(\Phi_B,\Int)
  \,, \qquad
  \Phi_B := \frac{1}{2\pi} \int_{\Real^2} {} ^{*\!}B (x) \, \der x
  \,.
\end{equation}
\end{Theorem}

Here and in the sequel, $ ^{*\!}B$~denotes the Hodge dual of~$B$.
Note that the former is just the usual scalar field when $d=2$.

It follows from Theorem~\ref{Thm.main.pre}
that the presence of a non-trivial magnetic field
in the plane enlarges the decay rate by an extra factor determined
by the \emph{total magnetic flux}~$\Phi_B$.
The lower bound $\gamma_B \geq (1+\beta)/2$
has been already established in~\cite{K7}.
In this paper we show that there is actually an equality,
as conjectured in~\cite[Sec.~4]{K7}.

However, the main result of this paper is the claim of Theorem~\ref{Thm.main.pre}
for the higher dimensions, stating that the transient effect of
the magnetic field is in fact undetectable on the level of
the polynomial decay rate~\eqref{rate}.
The extra decay of~\eqref{semigroup}
with respect to the magnetic-free case
must be therefore weaker than polynomial.
This result is somewhat surprising,
because the shifted Schr\"odinger operator~\eqref{Schrodinger}
exhibits certain similarities with
the two-dimensional magnetic Laplacian,
but it follows that it is actually very different.

We prove Theorem~\ref{Thm.main.pre} as a corollary of another theorem,
which gives an insight into the difference between two and higher dimensions
as regards~\eqref{Schrodinger}.
This result will be stated through the behaviour of the magnetic
field at (space) infinity.
For this reason it will be convenient to introduce spherical coordinates
\begin{equation}\label{spherical}
  \chart : S^{d-1} \times (0,\infty) \to \Real^d :
  \{(\sigma,r) \mapsto \sigma r\}
  \,.
\end{equation}
Then it is also natural to work in the \emph{Poincar\'e}
(or \emph{transverse}) gauge
\begin{equation}\label{Poincare}
  x \cdot A(x) = 0
\end{equation}
valid for all $x \in \Real^d$,
where the dot denotes the scalar product in~$\Real^d$.
Note that we can assume~\eqref{Poincare} without loss of any generality,
because of the gauge invariance of the physical theory.
Indeed, given a smooth tensor field~$B$,
the closedness $\der B=0$ ensures that the vector potential
\begin{equation}\label{Biot-Savart}
  A(x) := \int_0^1 x \cdot B(x u) \, u \, \der u
\end{equation}
satisfies both~\eqref{magnet.intro} and~\eqref{Poincare}.
We denote by $\mathsf{A} := \nabla\chart \cdot (A\circ\chart)$
the covariant counterpart of~$A$ in the spherical coordinates~\eqref{spherical}.
Since the last component of~$\mathsf{A}$ is zero due to~\eqref{Poincare},
we may think of $\sigma \mapsto \mathsf{A}(\sigma,r)$ for each fixed $r>0$
as a covariant vector field ($1$-form) on the sphere~$S^{d-1}$.
We introduce the quantity
\begin{equation}\label{nu.r}
  \nu_B(r) := \inf_{\stackrel[\varphi \not= 0 ]{}{\varphi \in H^1(S^{d-1})}}
  \frac{\displaystyle
  \int_{S^{d-1}}
  \big|\big(\der'- i \mathsf{A}(\sigma,r)\big)\varphi(\sigma)\big|_{S^{d-1}}^2
  \, \der\sigma}
  {\displaystyle
  \int_{S^{d-1}}  |\varphi(\sigma)|^2 \, \der\sigma}
  \,,
\end{equation}
where~$\der'$ denotes the exterior derivative on~$S^{d-1}$
and~$|\cdot|_{S^{d-1}}$ stands for the norm of a covariant vector on~$S^{d-1}$.
Obviously, $\nu_B(r)$~is the lowest eigenvalue of a magnetic
Laplace-Beltrami operator in $\sii(S^{d-1})$.
Assuming that~$B$ is smooth and compactly supported,
it follows from~\eqref{Biot-Savart} and~\eqref{spherical}
that the limit
\begin{equation}\label{limits}
  \mathsf{A}_\infty(\sigma) := \lim_{r \to \infty} \mathsf{A}(\sigma,r)
\end{equation}
exists as a smooth vector field from the unit sphere~$S^{d-1}$ to~$\Real^d$
and we may also define the corresponding number
\begin{equation}\label{limits.nu}
  \nu_B(\infty) := \lim_{r \to \infty} \nu_B(r)
  \,.
\end{equation}
Now we are in a position to state the following result.
\begin{Theorem}\label{Thm.main}
Let $d \geq 2$.
Suppose that~$B$ is smooth, closed and compactly supported.
Then
$$
  \gamma_B = \frac{1+\sqrt{\nu_B(\infty)}}{2}
  \,.
$$
\end{Theorem}

Theorem~\ref{Thm.main.pre} follows as a consequence of this unified identity.
Indeed, solving the spectral problem associated with~\eqref{nu.r} explicitly
(see, \eg, \cite{K7}), we find
\begin{equation}\label{d.2}
  \nu_B(\infty) = \dist(\Phi_B,\Int)^2
  \qquad\mbox{if}\qquad d=2
  \,.
\end{equation}
On the other hand, in higher dimensions we have the following equivalences.
\begin{Proposition}\label{Prop.eq.r}
Let $d \geq 3$.
Suppose that~$B$ is smooth, closed and compactly supported.
The following statements are equivalent, where $r \in (0,\infty)$ is fix.
\begin{enumerate}
\item[\emph{(i)}]
$\nu_B(r) = 0$.
\item[\emph{(ii)}]
The system
$\der'\varphi-i\mathsf{A}(\cdot,r)\varphi = 0$ on $S^{d-1}$
admits a smooth solution $\varphi \not=0$.
\item[\emph{(iii)}]
$\mathsf{A}(\cdot,r)$ is exact on~$S^{d-1}$,
\ie\ $\mathsf{A}(\cdot,r)=\der' f$ for some smooth function~$f$ on~$S^{d-1}$.
\item[\emph{(iv)}]
$\mathsf{A}(\cdot,r)$ is closed on~$S^{d-1}$,
\ie\ $\mathsf{B}'(\cdot,r) := \der' \mathsf{A}(\cdot,r) = 0$
as a $2$-covariant tensor on~$S^{d-1}$.
\item[\emph{(v)}]
The $S^{d-1}$ Hodge dual satisfies \,$ ^{*\!}\mathsf{B}'(\cdot,r) = 0$.
\item[\emph{(vi)}]
The $\Real^{d}$ Hodge dual satisfies
\,$ ^{*\!}\mathsf{B}^{\lambda_{1} \dots \lambda_{d-3} d}(\cdot,r) = 0$
for every $\lambda_{1}, \dots, \lambda_{d-3} \in \{1,\dots,d-1\}$.
\end{enumerate}

\end{Proposition}

If $d=3$, ${} ^{*\!}B$~is just the usual contravariant vector field
and~(vi) can be written in a coordinate-free version
\begin{equation*}
  {} ^{*\!}B(x) \cdot x = 0
\end{equation*}
for $|x|=r$ and all $\sigma \in S^{d-1}$.
In any case, assuming that~$B$ is compactly supported,
it follows from (iv)--(vi) that
\begin{equation}\label{d.high}
  \nu_B(\infty) = 0
  \qquad\mbox{if}\qquad d \geq 3
  \,.
\end{equation}
Using~\eqref{d.2} and~\eqref{d.high},
we therefore deduce Theorem~\ref{Thm.main.pre} from Theorem~\ref{Thm.main}.

For the reader not familiar with the concept of differential forms on manifolds,
we recall basic notions in Section~\ref{Sec.Pre}
together with giving a proof of Proposition~\ref{Prop.eq.r}.
Here we only remark that the equivalence between~(iii) and~(iv) fails when $d=2$,
because~$S^1$ is not simply connected, \cf~Remark~\ref{Rem.d2}.
This makes the two-dimensional situation intrinsically different.

To prove Theorem~\ref{Thm.main},
we adapt the method of self-similar variables,
which was developed for the heat equation by Escobedo and Kavian
in~\cite{Escobedo-Kavian_1987} and~\cite{MR959221}.
The technique was subsequently applied to convection-diffusion equations
by Escobedo, V\'{a}zquez and Zuazua in~\cite{MR1266100} and~\cite{MR1233647};
to the heat equation with the inverse-square potential
by V\'{a}zquez and Zuazua in~\cite{Vazquez-Zuazua_2000};
to the heat equation in twisted domains
by Krej\v{c}i\v{r}\'ik and Zuazua in~\cite{KZ1} and~\cite{KZ2};
to the present problem when $d=2$ by Krej\v{c}i\v{r}\'ik in~\cite{K7};
and, most recently,
to the heat equation in curved manifolds
by Kolb and Krej\v{c}i\v{r}\'ik in~\cite{KKolb}.
The present work can be considered as an extension of~\cite{K7} to any dimension,
but the presence of the inverse-square potential in~\eqref{Schrodinger}
also invokes~\cite{Vazquez-Zuazua_2000}.
We remark that the presence of magnetic Hardy inequalities
is essentially used in our study of the large-time behaviour
of the heat semigroup~\eqref{semigroup}
and the method thus represents an interesting application
of this functional-analytic tool.

The paper is organised as follows.
In the preliminary Section~\ref{Sec.Pre}
we collect a necessary material about the magnetic field
in any dimension and in spherical coordinates,
and establish Proposition~\ref{Prop.eq.r}.
We also give a precise definition of the magnetic
Schr\"o\-dinger operator~\eqref{Schrodinger}
and comment on a proof of~\eqref{spectrum}.
Theorem~\ref{Thm.Hardy} and other types of magnetic Hardy inequalities
are established in Section~\ref{Sec.Hardy}.
In Section~\ref{Sec.heat} we develop the method
of self-similar variables for~\eqref{Cauchy}
and reduce the large-time behaviour of the semigroup~\eqref{semigroup}
to a spectral analysis of a Schr\"odinger operator
with a singularly scaled magnetic field.
The latter is studied in Section~\ref{Sec.singular},
where we eventually give a proof of Theorem~\ref{Thm.main}.
The main ingredient in the spectral approach is Theorem~\ref{Prop.strong}
that establishes a norm-resolvent convergence of
the singularly scaled Schr\"odinger operators
to an Aharonov-Bohm-type operator.
The norm-resolvent convergence is obtained with help of
an abstract criterion (Lemma~\ref{Lem.nrs})
that we formulate and prove in Appendix~\ref{Sec.App}.
In Theorem~\ref{Prop.strong}, which we believe is of independent interest,
we employ among other things the magnetic Hardy inequality
of Theorem~\ref{Thm.Hardy.LW}.
The paper is concluded in Section~\ref{Sec.end}
by referring to some open problems.

\section{The magnetic field}\label{Sec.Pre}
%
In this preliminary section we collect some basic facts
about the concept of magnetic field in any dimension
and in spherical coordinates.
We refer, \eg, to~\cite{LR} and~\cite{Spivak1}
for notions related to tensors and differential forms.

\subsection{The magnetic potential, tensor and induction}\label{geom_prelim}
The magnetic field in the Euclidean space~$\Real^d$ with any $d \geq 2$
is most straightforwardly introduced through a 1-form $A = A_j \, \der x^j$,
where $A_j:\Real^d \to \Real$ are smooth functions and
$\der x^1,\dots,\der x^d$ is the dual basis
to the coordinate basis $\partial/\partial{x^1}, \dots, \partial/\partial{x^d}$
corresponding to the Cartesian coordinates $x = (x^1,\dots,x^d) \in \Real^d$.
Here and in the sequel we assume the Einstein summation convention,
with the range of Latin indices being~$1,\dots,d$.
Hence, $A$~is just a covariant vector field in~$\Real^d$.
In the Cartesian coordinates, $A$~coincides with
the contravariant vector field $A^j \partial/\partial{x^j}$.

Given a smooth 1-form~$A$, we introduce a $2$-form~$B$
as the exterior derivative of~\eqref{magnet.intro}.
The form $B$~can be identified with
a covariant skew-symmetric tensor field of order~$2$
with coefficients $B_{jk}=A_{k,j} - A_{j,k}$,
where we have introduced the comma notation for partial derivatives
(\ie\  $A_{k,j}:=\partial A_k/\partial x^j$).
$B$~is smooth in the sense that its coefficients are smooth.

The identity~\eqref{magnet.intro} means that~$B$ is an \emph{exact} form.
Hence, $B$~is necessarily \emph{closed}, \ie~\eqref{Gauss} holds.
Conversely, given a smooth $2$-form~$B$ satisfying~\eqref{Gauss},
we know that it is exact by the Poincar\'e lemma.
(Indeed, $\Real^d$ is clearly contractible.)
That is, there exists a smooth $1$-form~$A$
such that~\eqref{magnet.intro} holds.

Summing up, the correspondence~\eqref{magnet.intro} between~$A$ and~$B$
is consistent (\ie~one quantity can be obtained from the other in both directions)
provided that the latter satisfies~\eqref{Gauss}.
However, $B$~is ``more physical'' since it is uniquely determined
and appears in the Maxwell equation~\eqref{Gauss}.
In this physical context, $A$ and~$B$
are referred to as the \emph{magnetic potential}
and the \emph{magnetic tensor}, respectively.

Finally, we introduce the \emph{magnetic induction} $ ^{*\!}B$
as the Hodge-star dual of~$B$, \ie,
\begin{equation}\label{Hodge}
   ^{*\!}B = *\;\!\der A = {} ^{*\!}B^{l_1 \dots l_{d-2}} \
  \frac{\partial}{\partial{x^{l_1}}} \otimes \dots \otimes
  \frac{\partial}{\partial{x^{l_{d-2}}}}
  \,, \qquad \mbox{where} \qquad
  {} ^{*\!}B^{l_1 \dots l_{d-2}}
  = \frac{1}{2!} \, \eps^{l_1 \dots l_{d-2} j k} \, B_{jk}
  \,.
\end{equation}
Here $\eps$ is the Levi-Civita tensor,
which coincides with the usual Levi-Civita permutation symbol
in the Cartesian coordinates.
Note that~$ ^{*\!}B$ is a contravariant tensor field of order $d-2$.
Hence, $ ^{*\!}B$ is just a contravariant vector field in $d=3$
(it is a scalar field in $d=2$),
where it corresponds to the familiar quantity related to~$A$
via $ ^{*\!}B = \rot A$.

\subsection{The gauge invariance and the Poincar\'e gauge}
The fact that $A$~is not uniquely determined by~$B$,
is the well-known \emph{gauge invariance} of magnetic field.
Mathematically, one can employ this freedom
to work in a suitable choice (\emph{gauge}) of~$A$.

In components, condition~\eqref{Gauss} means that
the following Jacobi identity
\begin{equation}\label{Gauss.component}
  B_{kl,j} + B_{lj,k} + B_{jk,l} = 0
\end{equation}
holds for all indices $j,k,l \in \{1,\dots,d\}$.
From the skew-symmetry of~$B_{jk}$
and symmetries of Christoffel's symbols,
the partial derivatives in~\eqref{Gauss.component}
can be replaced by covariant derivatives (denoted by a semicolon here).
Then it is easy to see that~\eqref{Gauss.component} is equivalent
to the divergence-type identity
\begin{equation}\label{Maxwell}
  {} ^{*\!}B^{l_1 \dots l_{d-3} l_{d-2}
  }_{\phantom{l_1 \dots l_{d-3} l_{d-2}}; l_{d-2}} = 0
\end{equation}
for all indices $l_1 \dots l_{d-3} \in \{1,\dots,d\}$.
In $d=3$, this requirement reduces to the familiar formula
$\divergence {} ^{*\!}B = 0$.
Note also that~\eqref{Maxwell} is automatically satisfied in $d=2$,
where $ ^{*\!}B$ is a scalar field.

Assuming that~$B$ is smooth and closed and using~\eqref{Gauss.component},
it is straightforward to check that
the magnetic potential~$A$ defined by~\eqref{Biot-Savart}
satisfies~\eqref{Poincare} and~\eqref{magnet.intro}.
Note that, in components, \eqref{Biot-Savart} reads
\begin{equation}\label{Biot-Savart.component}
  A_j(x) = \int_0^1 x^l B_{lj}(x u) \, u \, \der u
  \,.
\end{equation}

A characteristic assumption of this paper
is that~$B$ is compactly supported.
It follows that the magnetic potential~$A$
in the Poincar\'e gauge~\eqref{Biot-Savart}
vanishes at infinity, too. Indeed,
\begin{equation}\label{A.vanish}
  |A(x)| \leq \frac{R^2 \, \|B\|_\infty}{|x|}
\end{equation}
for all $x \in \Real^d$ outside a big ball $D_R \supset \supp |B|$,
where~$|B|$ denotes the operator norm of~$B$
and $\|B\|_\infty := \sup_{x\in\Real^d} |B(x)|$.

\subsection{Spherical coordinates}
In spherical coordinates~\eqref{spherical},
the magnetic potential and magnetic tensor
are respectively given by
$$
  \mathsf{A} := \nabla\chart \cdot (A\circ\chart)
  \qquad\mbox{and}\qquad
  \mathsf{B} := \nabla\chart \cdot (B\circ\chart)
  \cdot (\nabla\chart)^T
  \,,
$$
where the transfer (or Jacobian) matrix reads
\begin{equation}\label{transfer}
  \nabla\chart =
  \begin{pmatrix}
    r \, \nabla'\sigma \\
    \sigma
  \end{pmatrix}
  \,.
\end{equation}
Here we use a concise notation where~$\nabla'$ is the gradient
with respect to local coordinates $\theta^1,\dots,\theta^{d-1}$
on the sphere~$S^{d-1}$.

As usual for curvilinear coordinates,
it is important to distinguish between covariant
and contravariant components of the tensors~$\mathsf{A}$ and $\mathsf{B}$.
The corresponding identification is given by the metric tensor
\begin{equation}\label{metric.spherical}
  g := \nabla\chart \cdot (\nabla\chart)^T
  = r^2 \, \der\sigma^2 + \der r^2
  =
  \begin{pmatrix}
    r^2 \, \gamma & 0 \\
    0 & 1
  \end{pmatrix}
  \,,
  \qquad
  |g| := \det(g) = r^{2(d-1)} \;\! |\gamma|
  \,,
\end{equation}
where
$
  \der\sigma^2=\gamma_{\mu\nu}(\theta)
  \, \der\theta^{\mu} \otimes \der\theta^{\nu}
$
is the metric of~$S^{d-1}$.
The range of Greek indices is assumed to be $1,\dots,d-1$.
As usual, we denote by~$g^{jk}$ the coefficients
of the inverse matrix~$g^{-1}$.
We shall not need explicit formulae for~$\gamma$ and~$\nabla'\sigma$,
but it is essential to realise that these quantities
are independent of the radial coordinate~$r$.

Formulae analogous to~\eqref{magnet.intro}, \eqref{Gauss} and~\eqref{Hodge}
hold for the spherical variables $q=(q',q^d)$
with $q' \in S^{d-1}$ and $q^d \in (0,\infty)$ as well;
it is just enough to replace $x, A, B$ with $q, \mathsf{A}, \mathsf{B}$.
In the last formula of~\eqref{Hodge},
it is important that we have introduced~$\eps$ as a tensor;
in spherical coordinates we thus have
$\eps^{l_1 \dots l_{d-2} j k} = |g|^{-1/2} \delta^{l_1 \dots l_{d-2} j k}$,
where~$\delta$ is the standard Levi-Civita permutation symbol (tensor density).
We obviously have $q^d=r=|x|$
and $q'=\sigma=x/|x|=\partial/\partial_r$, with $x \in \Real^d$.
Formulae~\eqref{Gauss.component} and~\eqref{Maxwell}
remain true in the spherical coordinates, too, after the replacement above.

The gauge formula~\eqref{Biot-Savart.component}
in the spherical coordinates reads
\begin{equation}\label{Biot-Savart.spherical}
  \mathsf{A}_\mu(\sigma,r)
  = \int_0^1 r \, \mathsf{B}_{d\mu}(\sigma,ru) \, u \, \der u
  = \int_0^r \frac{\mathsf{B}_{d\mu}(\sigma,v)}{r} \, v \, \der v
\end{equation}
for $\mu \in \{1,\dots,d-1\}$, while $\mathsf{A}_d=0$.
Passing back to the Cartesian coordinates
on the right hand side of~\eqref{Biot-Savart.spherical}
with help of~\eqref{transfer},
we get
\begin{equation}\label{passing}
  \mathsf{A}_\mu(\sigma,r)
  = \int_0^r \sigma^j B_{jk}(\sigma v) \, \sigma^k_{,\mu} \, v \, \der v
  = \int_0^r
  \left[
  \sigma \cdot B(\sigma v) \cdot (\nabla'\sigma)^T
  \right]_\mu
  v \, \der v
  \,.
\end{equation}
Hence, $\mathsf{A}_\mu(\sigma,r)$ depends on~$r$
only through the limit value in the integral
on the right hand side of this formula.
Assuming that~$B$ is compactly supported
(in the sense of its coefficients),
we thus see that there exists $R>0$ such that
$\mathsf{A}(\sigma,r)=\mathsf{A}(\sigma,R)$
for all $r \geq R$.
In particular, the limit~\eqref{limits} is well defined and
\begin{equation*}
  (\mathsf{A}_{\infty})_\mu(\sigma)
  = \int_0^\infty
  \left[
  \sigma \cdot B(\sigma v) \cdot (\nabla'\sigma)^T
  \right]_\mu
  v \, \der v
\end{equation*}
is obviously a smooth vector field
(in the sense of its coefficients).
On the other hand, the contravariant version of~$\mathsf{A}_\infty$
is a singular field; in fact,
$$
  |\mathsf{A}_\infty(\sigma)| = \frac{|\mathsf{A}_\infty(\sigma)|_{S^{d-1}}}{r}
  \,,
$$
which follows from the definitions
$
  |\mathsf{A}|^2 := \mathsf{A}_j g^{jk} \mathsf{A}_k
$
and
$
  |\mathsf{A}|_{S^{d-1}}^2 := \mathsf{A}_\mu \gamma^{\mu\nu} \mathsf{A}_\nu
$
and~\eqref{metric.spherical}.

\subsection{Proof of Proposition~\ref{Prop.eq.r}}\label{Sec.Prop.proof}
After the geometric preliminaries,
we are eventually in a position to establish 
the equivalent statements of Proposition~\ref{Prop.eq.r}.
\begin{proof}[Proof of Proposition~\ref{Prop.eq.r}]
\underline{(i) $\Leftrightarrow$ (ii)}.
(ii)~clearly implies~(i).
Since the embedding $H^1(S^{d-1}) \hookrightarrow \sii(S^{d-1})$
is compact, the infimum~\eqref{nu.r} is achieved
by a non-trivial function $\varphi \in H^1(S^{d-1})$.
Moreover, $\nu_B(r)$~is the first eigenvalue of
the self-adjoint operator $(-i\nabla_{\!\sigma} - \mathsf{A}(\sigma,r))^2$
in $\sii(S^{d-1})$, whose eigenfunctions are smooth
by elliptic regularity theory.
Hence, if $\nu_B(r)=0$, the numerator of~\eqref{nu.r}
must vanish with a non-trivial
smooth function~$\varphi$, which implies~(ii).
\\
\underline{(ii) $\Leftrightarrow$ (iii)}.
If $\mathsf{A}(\cdot,r)=\der' f$, then $\varphi = \Euler^{if}$ solves
the required system of differential equations.
Conversely, let~(ii) hold.
Multiplying the equation~$\varphi$ satisfies with~$\bar\varphi$
and combining the resulting equation with its complex-conjugate analogue,
we deduce $\der'|\varphi|^2=0$.
Hence, the magnitude $\rho := |\varphi|$ is constant on~$S^{d-1}$.
In particular, $\rho$ is positive because~$\varphi$ is non-trivial.
Inserting $\varphi = \rho \Euler^{if}$ with a real-valued function~$f$
into the equation~$\varphi$ satisfies,
we then obtain that $\der' f = \mathsf{A}(\cdot,r)$, which gives~(iii).
\\
\underline{(iii) $\Leftrightarrow$ (iv)}.
Any exact form is necessarily closed.
The opposite implication is non-trivial
(and in fact false for higher-order forms in general).
But all closed 1-forms on a simply connected manifold are exact
(see, \eg, \cite[Thm.~15.17]{MR1930091}).
Note that this argument differs from the Poincar\'e lemma
which requires that the manifold is contractible
(which does not hold for spheres).
\\
\underline{(iv) $\Leftrightarrow$ (v)}.
This equivalence follows from the duality relation~\eqref{Hodge},
which reads in the present situation
$$
  ^{*\!}\mathsf{B}'^{\lambda_1 \dots \lambda_{d-3}}
  = \frac{1}{2!} \, \eps^{\lambda_1 \dots \lambda_{d-3} \mu \nu}
  \, \mathsf{B}_{\mu\nu}
  \,.
$$
\\
\underline{(v) $\Leftrightarrow$ (vi)}.
Finally, using properties of the Levi-Civita tensor,
we observe the identity
$$
  ^{*\!}\mathsf{B}^{\lambda_{1} \dots \lambda_{d-3} d}
  = \frac{1}{2!} \, \eps^{\lambda_1 \dots \lambda_{d-3} d \mu \nu}
  \, \mathsf{B}_{\mu\nu}
  = \frac{1}{2!} \, \eps^{\lambda_1 \dots \lambda_{d-3} \mu \nu d}
  \, \mathsf{B}_{\mu\nu}
  = \frac{1}{2!} \, \eps^{\lambda_1 \dots \lambda_{d-3} \mu \nu}
  \, \mathsf{B}_{\mu\nu}
  = {} ^{*\!}\mathsf{B}'^{\lambda_1 \dots \lambda_{d-3}}
  \,,
$$
which proves the desired equivalence.
\end{proof}

Property~(iv) is particularly convenient,
since it reduces to a verification of the integrability conditions
$\mathsf{A}_{\nu,\mu} = \mathsf{A}_{\mu,\nu}$
for every $\mu,\nu \in \{1,\dots,d-1\}$.
On the other hand, property~(vi) is probably most physically intuitive,
since it says that a radial projection of the magnetic induction should vanish.
Note also that $ ^{*\!}\mathsf{B}'$ is just a scalar field on~$S^{d-1}$ if $d=3$.

\begin{Remark}[Proposition~\ref{Prop.eq.r} in $d=2$]\label{Rem.d2}
The two-dimensional situation is excluded from the proposition,
because there we do not have the equivalence between~(iii) and~(iv)
(only (iii)~$\Rightarrow$~(iv) holds in general).
Indeed, $S^1$~is not simply connected.
However, we still have equivalences among~(i), (ii) and~(iii).
It follows from the analysis in~\cite{K7}
that $\nu_B(r) = \dist(\Phi_B(r),\Int)^2$,
where~$\Phi_B(r)$ is the magnetic flux
$
  \Phi_B(r) := \frac{1}{2\pi} \int_{D_r} {} ^{*\!}B(x) \, \der x
$
in the ball $D_r$ of radius~$r$ centred at~$0$.
Hence, the exactness of $\mathsf{A}(\cdot,r)$ on~$S^1$
is rather determined by global properties of~$B$.
\end{Remark}

\subsection{The magnetic Schr\"odinger operator}\label{Sec.operator}
Recall the basic relation $B = \der A$,
where the magnetic potential is assumed to be smooth.
We introduce~\eqref{Schrodinger} as the Friedrichs extension
of the operator initially defined on~$C_0^\infty(\Real^d)$.
More specifically, $H_B$~is the self-adjoint
operator in $\sii(\Real^d)$ associated with the quadratic form
\begin{equation}\label{form}
  h_B[\psi] := \int_{\Real^d} \big|(\nabla-iA)\psi(x)\big|^2 \, \der x
  - c_d \int_{\Real^d} \frac{|\psi(x)|^2}{|x|^2} \, \der x
  \,, \qquad
  \Dom(h_B) := \overline{C_0^\infty(\Real^d)}^{\|\cdot\|_{h_B}}
  \,.
\end{equation}
Here the norm with respect to which the closure
is taken is defined by
\begin{equation}\label{norm}
  \|\psi\|_{h_B} := \sqrt{h_B[\psi] + \|\psi\|_{\sii(\Real^d)}^2}
  \,.
\end{equation}
Note that~$h_B$ is non-negative due to
the diamagnetic inequality~\eqref{diamagnetic}
and the classical Hardy inequality~\eqref{Hardy}.

At a first sight, we simply remark that
$\Dom(h_B) \supset H^1(\Real^d)$ whenever~$A$ is bounded.
On the other hand, it is known that $\Dom(h_0)$ is strictly larger than $H^1(\Real^d)$.
For that it is enough to consider functions
which behave at the origin $x=0$ like:
$$
  \psi_\alpha(x)\sim |x|^{-(d-2)/2}\left(\log \frac{1}{|x|}\right)^{\alpha}
  \,, \qquad
  -1/2\leq \alpha< 1/2
  \,.
$$
Then it is not difficult to check that
$\psi_\alpha \in \Dom(h_0)\setminus H^1(\Real^d)$
(see, \eg, \cite[Sec.~2.2]{Vazquez-Zagraphopoulos-2012}).
(As a matter of fact,
the authors in~\cite{Vazquez-Zagraphopoulos-2012}
only pointed out the cases $0< \alpha< 1/2$,
because such $\psi_\alpha$ are the most singular,
but it is easy to extend the argument for $\alpha \leq 0$.)
%
Using arguments as in the proof of Lemma~\ref{Lem.norm.eq} below,
it follows that $\Dom(h_0)=\Dom(h_B)$ provided that~$A$ is bounded.
%

%
\begin{Lemma}\label{Lem.core}
Let $d \geq 2$.
Suppose that~$B$ is smooth and closed.
Then $C_0^\infty(\Real^d \setminus \{0\})$ is a core of~$h_B$.
\end{Lemma}
\begin{proof}
It is enough to show that for any $\psi \in C_0^\infty(\Real^d)$
there exists a family of functions $\psi_\delta \in C_0^\infty(\Real^d \setminus \{0\})$
such that $\|\psi-\psi_\delta\|_{h_B} \to 0$ as $\delta \to 0$.
For this approximation family one can take for instance that
of \cite[proof of Corol.~VIII.6.4]{Edmunds-Evans}.
We leave the details to the reader.
\end{proof}

It is well known that for different magnetic potentials
whose exterior derivative yields the same magnetic tensor
the corresponding operators are unitarily equivalent.
Consequently, the spectrum as well as the validity of 
the Hardy inequality of Theorem~\ref{Thm.Hardy}
and the decay rate~\eqref{rate} of Theorem~\ref{Thm.main.pre}
do not depend on the particular choice of the magnetic potential.

Using the gauge freedom, in this paper
we often (but not exclusively)
choose the Poincar\'e gauge of~\eqref{Biot-Savart}.
This choice is convenient because we wish to work
in the spherical coordinates~\eqref{spherical}
in which the radial component~$\mathsf{A}_d$ vanishes.
We introduce the unitary transform
\begin{equation}\label{unitary.spherical}
  \mathcal{U}: \sii(\Real^d) \to
  \sii\big(S^{d-1}\times(0,\infty),r^{d-1} \;\! \der\sigma \, \der r\big) :
  \left\{\psi \mapsto \psi\circ\chart\right\}
  \,,
\end{equation}
where~$\der\sigma$ is the volume element of~$S^{d-1}$. Then
$H_B$~is unitarily equivalent to the operator
$\mathsf{H}_B := \mathcal{U} H_B \mathcal{U}^{-1}$
in $\sii\big(S^{d-1}\times(0,\infty),r^{d-1} \;\! \der\sigma \, \der r\big)$,
which is associated with the quadratic form
$\mathsf{h}_B[\phi] := h_B[\mathcal{U}^{-1}\phi]$,
$\Dom(\mathsf{h}_B) := \mathcal{U} \Dom(h_B)$.
Using~\eqref{metric.spherical}
and recalling the notation~$\der'$ for the exterior derivative
on the sphere~$S^{d-1}$,
we have
\begin{align}\label{form.spherical}
  \mathsf{h}_B[\phi]
  &= \int_{S^{d-1}\times(0,\infty)}
  \left[
  \frac{\big|(\der' - i\mathsf{A})\phi\big|_{S^{d-1}}^2}{r^2}
  + |\phi_{,r}|^2
  - c_d \, \frac{|\phi|^2}{r^2}
  \right]
  r^{d-1} \;\! \der\sigma \, \der r
  \,.
\end{align}
Here we prefer to write $\phi_{,r}=\partial\phi/\partial r$
instead of $\phi_{,d}=\partial\phi/\partial q^d$.
Also, hereafter we usually suppress the arguments on which the functions depend.

We conclude this section by commenting on a proof of~\eqref{spectrum}.
The fact that no negative point belongs to the spectrum of~$H_B$
follows from the diamagnetic inequality~\eqref{diamagnetic}
and the Hardy inequality~\eqref{Hardy}.
On the other hand, to show that every point in $[0,\infty)$
belongs to the spectrum of~$H_B$,
one can use the Weyl criterion,
namely its version adapted to quadratic forms in~\cite[Thm.~5]{KL}.

\section{The Hardy inequality}\label{Sec.Hardy}
%
In this section we give a proof of the magnetic Hardy inequality
of Theorem~\ref{Thm.Hardy}. We present two approaches,
where the first one does not yield Theorem~\ref{Thm.Hardy}
under the stated minimal assumptions,
but on the other hand, it provides
the constant~$c_{d,B}$ in a more explicit form through~$\nu_B$.
The basic idea of both the approaches
is to derive first a ``local'' Hardy inequality,
\ie~a version of~\eqref{Hardy.magnet} where the weight
in the integral on the right hand side is not necessarily
an everywhere positive function.

\subsection{An auxiliary result}
We shall essentially use the following one-dimensional inequalities.
\begin{Lemma}\label{Lem.aux}
Let $r_0 > 0$. There exists a positive constant~$\gamma$ depending on~$r_0$
such that for all $f \in C_0^\infty(\Real\setminus\{r_0\})$,
\begin{align}
  \int_0^{r_0}
  |f'(r)|^2 \, r \, \der r
  & \geq \gamma
  \int_0^{r_0} |f(r)|^2 \, r \, \der r
  \,,
  \label{aux1}
  \\
  \int_{r_0}^\infty
  |f'(r)|^2 \, r \, \der r
  & \geq \gamma
  \int_{r_0}^\infty \frac{|f(r)|^2}{r^2\log^2(r/r_0)} \, r \, \der r
  \,.
  \label{aux2}
\end{align}
\end{Lemma}

The inequalities are rather elementary and probably well known
(see, \eg, \cite{MR1605678} for a usage of the first estimate),
so we leave the proofs to the reader.
We note that the left hand sides of~\eqref{aux1} and~\eqref{aux2}
are just radial parts of the quadratic form of the two-dimensional Laplacian.

\subsection{The Poincar\'e gauge approach}
The first idea is to pass to the spherical coordinates~\eqref{spherical},
choose the Poincar\'e gauge~\eqref{Poincare}
and employ the definition of the function~$\nu_B$ given in~\eqref{nu.r}.
With help of Fubini's theorem, we thus obtain from~\eqref{form.spherical}
\begin{equation}\label{approach.Poincare}
  \mathsf{h}_B[\phi]
  \geq \int_{S^{d-1}\times(0,\infty)}
  \left[
  |\phi_{,r}|^2
  + \frac{\nu_B(r)-c_d}{r^2} \, |\phi|^2
  \right]
  r^{d-1} \;\! \der\sigma \, \der r
  \,,
\end{equation}
for any $\phi := \mathcal{U}\psi$,
where~$\psi$ is an arbitrary function from $C_0^\infty(\Real^d)$
and~$\mathcal{U}$ is the unitary transform~\eqref{unitary.spherical}.
We remark that $\nu_B(r) = \mathcal{O}(r^2)$ as $r \to 0$,
\cf~\eqref{passing},
so that $\nu_B(r)/r^2$ has actually no singularity at $r=0$.

Next we employ an elementary inequality ($d\geq 2$)
\begin{equation}\label{Hardy.1D}
  \forall \phi \in C_0^\infty(\Real)
  \,, \quad
  \int_0^\infty |\phi'(r)|^2 \, r^{d-1} \;\! \der r
  \geq c_d
  \int_0^\infty \frac{|\phi(r)|^2}{r^2} \, r^{d-1} \;\! \der r
  \,.
\end{equation}
It can be deduced from~\eqref{Hardy}
when written in the spherical coordinates
and applied to radially symmetric functions
(with help of a density argument to allow arbitrary values $\phi'(0)$), but it can be also proved directly.
Using~\eqref{Hardy.1D} in~\eqref{approach.Poincare}
and passing back to the Cartesian coordinates,
we conclude with the following local Hardy inequality.
\begin{Proposition}\label{Prop.Hardy.Poincare}
Let $d \geq 2$.
Suppose that~$B$ is smooth and closed.
Then
\begin{equation}\label{Hardy.Poincare}
  \forall \psi \in C_0^\infty(\Real^d) \,, \quad
  h_B[\psi]
  \geq
  \int_{\Real^d} \frac{\nu_B(|x|)}{|x|^2} \, |\psi(x)|^2 \, \der x
  \,.
\end{equation}
\end{Proposition}

A defect of this inequality is that~$\nu_B$ may vanish identically
even if $B \not= 0$, \cf~Proposition~\ref{Prop.eq.r}.
If this function is non-trivial, however,
the local inequality can be extended to the whole~$\Real^d$.
\begin{Theorem}\label{Thm.Hardy.Poincare}
Let $d \geq 2$.
Suppose that~$B$ is smooth and closed.
If $\nu_B\not=0$ (\ie~the function~$\nu_B$ is non-trivial),
then there exists a positive constant~$c_{d,B}$
such that~\eqref{Hardy.magnet} holds.
\end{Theorem}
\begin{proof}
By virtue of Lemma~\ref{Lem.core},
it is enough to prove~\eqref{Hardy.magnet}
for $\psi \in C_0^\infty(\Real^d \setminus \{0\})$.
Fixing such a function, we denote by $\phi:=\mathcal{U}\psi$
its counterpart in the spherical coordinates throughout the proof.
It follows from the variational definition~\eqref{nu.r}
that~$\nu_B$ is Lipschitz continuous.
Hence, the hypothesis ensures that there exists a positive constant~$\nu$
(depending on the behaviour of~$\nu_B$)
and a bounded open interval $I \subset (0,\infty)$ such that
$\nu_B(r) / r^2 \geq \nu > 0$ for all $r \in I$.
From Proposition~\ref{Prop.Hardy.Poincare} we thus conclude
\begin{equation}\label{step1}
  \forall \psi \in C_0^\infty(\Real^d) \,, \quad
  h_B[\psi]
  \geq
  \nu \int_{\Real^d} \chi_I(x) \, |\psi(x)|^2 \, \der x
  \,,
\end{equation}
where~$\chi_I$ denotes the characteristic function
of the spherical shell $\{x \in \Real^d \,|\, |x| \in I\}$.

To extend this local Hardy inequality to~$\Real^d$,
we employ the presence of the other terms that
we neglected when passing from~\eqref{approach.Poincare}
to~\eqref{Hardy.Poincare}
\begin{equation*}
  \mathsf{h}_B[\phi]
  \geq \int_{S^{d-1}\times(0,\infty)}
  \left[
  |\phi_{,r}|^2
  - \frac{c_d}{r^2} \, |\phi|^2
  \right]
  r^{d-1} \;\! \der\sigma \, \der r
  \,.
\end{equation*}
If $d=2$, $c_2=0$ and the right hand side is just
an integral of the derivative.
To obtain the same form for any $d \geq 3$,
we perform the standard Hardy transform $f := r^{(d-2)/2} \phi$
to obtain
\begin{equation}\label{Hardy.transform}
  \int_{S^{d-1}\times(0,\infty)}
  \left[
  |\phi_{,r}|^2
  - \frac{c_d}{r^2} \, |\phi|^2
  \right]
  r^{d-1} \;\! \der\sigma \, \der r
  = \int_{S^{d-1}\times(0,\infty)}
  |f_{,r}|^2
  \, r \, \der\sigma \, \der r
  \,.
\end{equation}
Denoting by~$r_0$ the middle point of~$I$,
we introduce a cut-off function
$\xi \in C^\infty((0,\infty))$
such that $|\xi| \leq 1$,
$\xi$~vanishes in a neighbourhood of~$r_0$
and $\xi=1$ outside the interval~$I$.
We keep the same notation~$\xi$ for
the function $1 \otimes \xi$ on $S^{d-1}\times(0,\infty)$.
Writing $f=\xi f + (1-\xi) f$
and using Lemma~\ref{Lem.aux} with help of Fubini's theorem
(\cf~\cite[proof of Thm.~3.1]{KZ1} for a similar estimate), 
we get
\begin{eqnarray*}
\lefteqn{
  \int_{S^{d-1}\times(0,\infty)}
  \frac{|f|^2}{1+r^2\log^2(r/r_0)} \, r \, \der\sigma \, \der r
}
\\
  &&\leq
  \frac{4}{\gamma} \int_{S^{d-1}\times (0,\infty)}
  |f_{,r}|^2 \ r \, \der\sigma \, \der r
  + \left( \frac{4 \|\xi'\|_\infty^2}{\gamma} + 2 \right)
  \int_{S^{d-1}\times I}
  |f|^2 \, r \, \der\sigma \, \der r
  \,.
\end{eqnarray*}
Here $\|\xi'\|_\infty$ is the supremum norm of the derivative of~$\xi$
as a function on~$(0,\infty)$.
Coming back to the test function~$\psi$,
we have therefore proved
\begin{equation}\label{step2}
  h_B[\psi]
  \geq \frac{\gamma}{4} \int_{\Real^d}
  \frac{|\psi(x)|^2}{1+|x|^2\log^2(|x|/r_0)} \, \der x
  - \left(\|\xi'\|_\infty^2+\frac{\gamma}{2}\right)
  \int_{\Real^d}
  \chi_I(x) \, |\psi(x)|^2 \, \der x
  \,.
\end{equation}

Finally, combining~\eqref{step1} and~\eqref{step2},
we get
$$
  h_B[\psi]
  \geq
  \left[
  (1-\eps) \nu - \eps \left(\|\xi'\|_\infty^2+\frac{\gamma}{2}\right)
  \right]
  \int_{\Real^d} \chi_I(x) \, |\psi(x)|^2 \, \der x
  + \eps \, \frac{\gamma}{4} \int_{\Real^d}
  \frac{|\psi(x)|^2}{1+|x|^2\log^2(|x|/r_0)} \, \der x
$$
with any $\eps > 0$.
Choosing~$\eps$ in such a way that the square bracket vanishes,
we obtain~\eqref{Hardy.magnet} with
$$
  c_{d,B} \geq
  \frac{\displaystyle \frac{\gamma}{4} \, \nu}
  {\displaystyle \nu+\|\xi'\|_\infty^2+\frac{\gamma}{2}} \
  \inf_{r \in (0,\infty)} \frac{1+r^2\log^2(r)}{1+r^2\log^2(r/r_0)}
  >0
  \,.
$$
The theorem is proved.
\end{proof}

Assuming instead of $\nu_B \not= 0$ the stronger hypothesis
that~$\nu_B$ is ``non-trivial at infinity'', \ie~$\nu_B(\infty)\not=0$,
we can get rid of the logarithm on the right hand side of~\eqref{Hardy}.
\begin{Theorem}[Laptev and Weidl~\cite{Laptev-Weidl_1999}]\label{Thm.Hardy.LW}
Let $d \geq 2$.
Suppose that~$B$ is smooth, closed and compactly supported.
If $\nu_B(\infty)\not=0$,
then there exists a positive constant~$\tilde{c}_{d,B}$ such that
for any smooth~$A$ satisfying $\der A = B$,
the following inequality holds
\begin{equation}\label{Hardy.magnet.LW}
  \forall \psi \in C_0^\infty(\Real^d) \,, \quad
  \int_{\Real^d}
  |(\nabla-iA)\psi(x)|^2 \, \der x
  - c_d
  \int_{\Real^d} \frac{|\psi(x)|^2}{|x|^2} \, \der x
  \geq \tilde{c}_{d,B}
  \int_{\Real^d} \frac{|\psi(x)|^2}{1+|x|^2} \, \der x
  \,.
\end{equation}
\end{Theorem}
\begin{proof}
Let $\psi \in C_0^\infty(\Real^d)$.
By Theorem~\ref{Thm.Hardy.Poincare}, we have
$$
  h_B[\psi]
  \geq
  c_{d,B} \int_{\Real^d} \frac{|\psi(x)|^2}{1+|x|^2\log^2(|x|)} \, \der x
  \geq
  c_{d,B} \int_{D_R} \frac{|\psi(x)|^2}{1+|x|^2\log^2(|x|)} \, \der x
  \geq
  c_{d,B} \, a_R \int_{D_R} \frac{|\psi(x)|^2}{1+|x|^2} \, \der x
$$
for any ball~$D_R$ of radius~$R$ centred at the origin, where
$$
  a_R := \inf_{r\in(0,R)} \frac{1+r^2}{1+r^2\log^2(r)}
$$
is obviously a positive constant. 	
At the same time, Proposition~\ref{Prop.Hardy.Poincare} yields
$$
  h_B[\psi]
  \geq
  \nu_B(\infty)
  \int_{\Real^d \setminus D_R} \frac{|\psi(x)|^2}{|x|^2} \, \der x
  \geq
  \nu_B(\infty)
  \int_{\Real^d \setminus D_R} \frac{|\psi(x)|^2}{1+|x|^2} \, \der x
  \,,
$$
where~$D_R$ is a ball containing the support of~$B$.
Combining these two inequalities, we get~\eqref{Hardy.magnet.LW} with
$$
  \tilde{c}_{d,B} \geq \frac{c_{d,B} \, a_R + \nu_B(\infty)}{2}
  > 0
  \,,
$$
where the best estimate is obtained for
$R := \sup\{|x| \,|\, x \in \supp|B|\}$.
\end{proof}
\begin{Remark}\label{Rem1}
Of course, Theorem~\ref{Thm.Hardy.LW} is void for $d \geq 3$,
where $\nu(\infty)=0$ by Proposition~\ref{Prop.eq.r}
and the compactness of the support of~$B$, \cf~\eqref{d.high}.
The only non-trivial situation is thus $d=2$,
where~\eqref{d.2} holds
and Theorem~\ref{Thm.Hardy.LW} is just a special case
of the celebrated magnetic Hardy inequality
of Laptev and Weidl established in~\cite[Thm.~1]{Laptev-Weidl_1999}.
\end{Remark}

\subsection{The gauge-free approach: proof of Theorem~\ref{Thm.Hardy}}
Following \cite[proof of Thm.~3.4]{Weidl_1999},
we start with the unitary transform
\begin{equation}\label{unitary.Weidl}
  \mathcal{V} : \sii(\Real^d) \to \sii\big(\Real^d,|x|^{-(d-2)} \,\der x\big):
  \left\{\psi \mapsto |x|^{(d-2)/2} \, \psi \right\}
  \,.
\end{equation}
It maps~$H_B$ into a unitarily equivalent operator
$T_B := \mathcal{V} H_B \mathcal{V}^{-1}$ in
$\sii\big(\Real^d,|x|^{-(d-2)} \,\der x\big)$,
which is associated with the quadratic form
$t_B[g] := h_B[\mathcal{V}^{-1} g]$, $\Dom(t_B):=\mathcal{V}\Dom(h_B)$.
By definition, $\Dom(t_B)$ is the closure of $C_0^\infty(\Real^d)$
with respect to the norm $\|\cdot\|_{t_B}$
which is defined in analogy with~\eqref{norm}.
By virtue of Lemma~\ref{Lem.core},
we could alternatively characterise $\Dom(t_B)$ through
the closure of $C_0^\infty(\Real^d \setminus \{0\})$.
On this more restricted core,
using
\begin{equation}\label{trick.Weidl}
  |(\nabla-iA)\mathcal{V}^{-1}g|^2
  = |x|^{-(d-2)}
  \left(
  |(\nabla-iA)g|^2 + c_d \, \frac{|g|^2}{|x|^2}
  - \frac{d-2}{2} \frac{x}{|x|^2} \cdot \nabla|g|^2
  \right)
\end{equation}
and integrating by parts with help of $\divergence(x/|x|^d)=0$,
it is straightforward to check the key identity
\begin{equation}\label{form.Weidl}
  t_B[g]
  = \int_{\Real^d}
  |(\nabla -iA) g(x)|^2 \, |x|^{-(d-2)} \, \der x
  \,.
\end{equation}
Since
$
  |x|^{-(d-2)} \, \der x = r \, \der\sigma \, \der r
$
in the spherical coordinates,
the right hand side of~\eqref{form.Weidl}
can be interpreted as a two-dimensional magnetic form.

Now we are inspired by the method
used in~\cite{K6-erratum.simple} and~\cite{KZ1}
to establish a Hardy-type inequality in twisted waveguides.
Instead of~$\nu_B$, we introduce a more global quantity
\begin{equation}\label{lambda}
  \mu_B(R)
  := \inf_{\stackrel[g \not= 0 ]{}{g \in C^\infty(\overline{D_R})}}
  \frac{\displaystyle
  \int_{D_R}
  \big|(\nabla- i A)g(x)\big|^2 \, |x|^{-(d-2)} \, \der x}
  {\displaystyle
  \int_{D_R}  |g(x)|^2 \, |x|^{-(d-2)} \,\der x}
  \,,
\end{equation}
where $D_R$ is the $d$-dimensional open ball of radius $R>0$
centred at the origin of~$\Real^d$
(we do not use the standard notation~$B_R$ for the ball
to avoid a confusion with the magnetic field~$B$).
$\mu_B(R)$ is the spectral threshold of
the self-adjoint operator~$T_B^R$
in $\sii\big(D_R,|x|^{-(d-2)} \,\der x\big)$
associated with the quadratic form
\begin{equation}\label{form.r}
  t_B^R[g] := \int_{D_R} \big|(\nabla-iA)g(x)\big|^2
  \, |x|^{-(d-2)} \,\der x
  \,, \qquad
  \Dom(t_B^R) :=
  \overline{C^\infty(\overline{D_R})}^{\|\cdot\|_{t_B^R}}
  \,,
\end{equation}
where $\|\cdot\|_{t_B^R}$ is defined in analogy with~\eqref{norm}.
Instead of the space $C^\infty(\overline{D_R})$,
we could take the closure of
restrictions of $C_0^\infty(\Real^d\setminus\{0\})$ to~$D_R$.

Since~$A$ is bounded on the ball $D_R$,
we clearly have $\Dom(t_B^R) = H^1(D_R)$ in $d=2$.
In any dimension, it is easy to see that
$\Dom(t_B^R) \subset H^1(D_R)$ as well as
$\sii\big(D_R,|x|^{-(d-2)} \,\der x\big) \subset \sii(D_R)$.
More importantly, employing the compactness of the embedding
$H^1(D_R) \hookrightarrow \sii(D_R)$ in two dimensions,
we may deduce from~\eqref{form.r} that
$\Dom(t_B^R)$ is compactly embedded
in $\sii\big(D_R,|x|^{-(d-2)} \,\der x\big)$.
Consequently, the infimum in~\eqref{lambda} is achieved
by a non-trivial function $g_1 \in \Dom(t_B^R)$
and $\mu_B(R)$ is just the first eigenvalue of~$T_B^R$.
$R \mapsto \mu_B(R)$ defines a continuous function on $(0,\infty)$.

The next result is an analogue of Proposition~\ref{Prop.Hardy.Poincare}
and follows directly from the definition~\eqref{lambda}
with help of the unitary equivalence of~$H_B$ and~$T_B$
through~\eqref{unitary.Weidl}.
\begin{Proposition}\label{Prop.Hardy.free}
Let $d \geq 2$.
Suppose that~$B$ is smooth and closed.
Then, for any $R>0$,
\begin{equation}\label{Hardy.free}
  \forall \psi \in C_0^\infty(\Real^d) \,, \quad
  h_B[\psi]
  \geq
  \mu_B(R)
  \int_{D_R} |\psi(x)|^2 \, \der x
  \,.
\end{equation}
\end{Proposition}

On the other hand, the following result is quite non-trivial
and makes the precedent proposition highly important
as a robust local Hardy inequality whenever $B \not= 0$.
\begin{Proposition}\label{Prop.Hardy.crucial}
Let $d \geq 2$.
Suppose that~$B$ is smooth and closed.
Then
$$
  \mu_B=0
  \qquad \Longleftrightarrow \qquad
  B = 0
  \,.
$$
\end{Proposition}
\begin{proof}
We are partially inspired by~\cite[proof of Prop.~2.1.3]{Fournais-Helffer_2009}.
If $B=0$, then we may take $A=0$ and consequently $\mu_0=0$,
with $g_1^0:=1$ being the first eigenfunction of~$T_0^R$ for any $R>0$.
Conversely, let us assume $\mu_B=0$.
For any fixed $R>0$, let~$g_1$ denote the first eigenfunction of~$T_B^R$.
By elliptic regularity theory, we know that~$g_1$ is smooth
in $\overline{D_R} \setminus \{0\}$.
The diamagnetic inequality~\eqref{diamagnetic}
and the assumption $\mu_B(R)=0$ imply
that the magnitude~$|g_1|$ is constant in~$D_R$.
We may assume $|g_1|=g_1^0=1$ and write $g_1 = \Euler^{i\varphi}$
with some real-valued function~$\varphi$ such that
$
  |\nabla\varphi| \in \sii\big(D_R,|x|^{-(d-2)} \,\der x\big)
$
which is in fact smooth in $\overline{D_R} \setminus \{0\}$.
From $t_B^R[\Euler^{i\varphi}]=0$ we then find that $\nabla\varphi = A$
in $\overline{D_R} \setminus \{0\}$.
That is, $A$~is exact and thus $B=\der A = 0$
in the punctured ball $\overline{D_R} \setminus \{0\}$.
Since this is true for any $R>0$, we conclude that $B=0$ in~$\Real^d$.
\end{proof}
\begin{Remark}
Combining Propositions~\ref{Prop.Hardy.free} and~\ref{Prop.Hardy.crucial},
we get a result reminiscent of Weidl's inequality
$$
  \forall \psi \in C_0^\infty(\Real^d) \,, \quad
  h_B[\psi] \geq c(d,A) \int_{D_2\setminus D_1} |\psi(x)|^2 \, \der x
$$
that he obtained in~\cite[Sec.~3.5]{Weidl_1999}
as a consequence of his more abstract results in~\cite{Weidl_1999a}
under the assumption that~$A$ is ``non-trivial''.
\end{Remark}

The following ultimate result is just Theorem~\ref{Thm.Hardy}.
\begin{Theorem}\label{Thm.Hardy.free}
Let $d \geq 2$.
Suppose that~$B$ is smooth and closed.
If $B\not=0$, then there exists a positive constant~$c_{d,B}$
such that~\eqref{Hardy.magnet} holds.
\end{Theorem}
\begin{proof}
Once we have~\eqref{Hardy.free} with a positive $\mu_B(R)>0$
(\cf~Proposition~\ref{Prop.Hardy.crucial}),
we have in particular~\eqref{step1} with $I=(0,R)$
and the global Hardy inequality
follows by mimicking the rest of the proof of Theorem~\ref{Thm.Hardy.Poincare}.
In particular, we obtain
$$
  c_{d,B} \geq
  \frac{\displaystyle \frac{\gamma}{4} \, \mu_B(R)}
  {\displaystyle \mu_B(R)+\|\xi'\|_\infty^2+\frac{\gamma}{2}}
  \ \inf_{r \in (0,\infty)} \frac{1+r^2\log^2(r)}{1+r^2\log^2(2r/R)}
  > 0
  \,,
$$
where~$\gamma$ is the constant from Lemma~\ref{Lem.aux}.
\end{proof}
\begin{Remark}
Because of the meaning of~$\mu_B(R)$,
let us mention that the lowest eigenvalue of
the magnetic Neumann Laplacian in domains has been
extensively studied in connection with superconductivity,
(see, \eg, \cite{Helffer-Morame_2001},
\cite{Fournais-Helffer_2006} and \cite{Fournais-Helffer_2007}).
\end{Remark}

%

\section{The heat equation}\label{Sec.heat}
%
In this section we reduce the proof of Theorem~\ref{Thm.main}
to a spectral analysis of a family of operators.

\subsection{The physical variables}
Since~$H_B$ is a self-adjoint operator (\cf~Section~\ref{Sec.operator}),
the semigroup~\eqref{semigroup} can be constructed by means of
the functional calculus.
Another possibility is to apply the standard semigroup theory;
according to the Hille-Yosida theorem~\cite[Thm.~7.7]{Brezis_new},
for any initial datum $u_0\in \sii(\Real^d)$,
there exists a unique solution
\begin{equation}\label{sol.smooth}
  u\in C^0\big([0, \infty); L^2(\Real^d)\big)
  \cap C^1\big((0, \infty); L^2(\Real^d)\big)
  \cap C^0\big((0,\infty); \Dom(H_B)\big)
 \end{equation}
of the evolution problem~\eqref{Cauchy}.
In particular,
\begin{equation}\label{space-valued}
  u\in L_\mathrm{loc}^2\big((0, \infty); \Dom(h_B)\big)
  \qquad \mbox{and} \qquad
  u_{,t}\in L_\mathrm{loc}^2\big((0, \infty); \Dom(h_B)^*\big)
  \,.
\end{equation}
With an abuse of notation, we denote by the same symbol~$u$
both the function of the space-time variables
$(x,t) \in \Real^d \times (0,\infty)$
and the Hilbert-space-valued mapping $u:(0,\infty) \to \sii(\Real^d)$.
In particular, we identify the partial derivative~$u_{,t}$
with respect to the time variable
with the Hilbert-space weak derivative~$u'$.

Later on we shall transfer~\eqref{Cauchy}
to a non-autonomous evolution problem,
to which the standard semigroup theory does not apply
and variational tools have to be used instead.
Let us therefore formulate already~\eqref{Cauchy} in this setting.
We say that the Hilbert-space-valued function~$u$
satisfying~\eqref{space-valued} is a \emph{weak solution} of~\eqref{Cauchy}
if
\begin{equation}\label{weaksol}
  \tensor[_{\Dom(h_B)}]{\big\langle \phi,
  u_{,t}(t)\big\rangle}{_{\Dom(h_B)^*}}
  +h_B\big(\phi, u(t)\big)=0
  \,,
\end{equation}
for each $\phi\in \Dom(h_B)$ and a.e.\ $t\in [0, \infty)$,
and $u(0)=u_0 \in \sii(\Real^d)$.
Here $h_B(\cdot, \cdot)$ denotes the sesquilinear form
associated to $h_B[\cdot]$ and
$
  \tensor[_{\Dom(h_B)}]{\langle \cdot, \cdot \rangle}{_{\Dom(h_B)^*}}
$
stands for the duality pairing of~$\Dom(h_B)$ and~$\Dom(h_B)^*$.
The existence and uniqueness of the weak solution~$u$ of~\eqref{weaksol}
follows by an abstract theorem of J.~L.~Lions~\cite[Thm.~10.9]{Brezis_new}.

Now we would like to restrict the initial data~$u_0$ of~\eqref{Cauchy}
to the weighted space $\sii_w(\Real^d)$ introduced in~\eqref{weighted}.
It corresponds to the weak formulation
\begin{equation}\label{weaksol.weighted}
  \tensor[_{\Dom(h_B^w)}]{\big\langle \phi,
  u_{,t}(t)\big\rangle}{_{\Dom(h_B^w)^*}}
  +h_B^w\big(\phi, u(t)\big)=0
  \,,
\end{equation}
where~$h_B^w(\cdot,\cdot)$ is the sesquilinear form associated
with the quadratic form
$$
  h_B^w[\psi] :=
  \big\|(\nabla-iA)\psi\big\|_{\sii_w(\Real^d)}^2
  - c_d \left\| \frac{\psi}{|x|} \right\|_{\sii_w(\Real^d)}^2
  \,, \qquad
  \Dom(h_B^w) := \overline{C_0^\infty(\Real^d)}^{\|\cdot\|_{h_B^w}}
  \,.
$$
Here the norm with respect to which the closure
is taken is defined in analogy with~$\eqref{norm}$.
We remark that~$h_B^w$ is non-negative.
Indeed, employing the trick~\eqref{trick.Weidl}, we get
$$
  h_B^w[\mathcal{V}^{-1}g]
  = \big\|(\nabla-iA)g\big\|_{\sii_{w\eta}(\Real^d)}^2
  + \frac{d-2}{4} \left\|g\right\|_{\sii_{w\eta}(\Real^d)}^2
$$
for every $g \in C_0^\infty(\Real^d \setminus\{0\})$,
where $\eta(x):=|x|^{-(d-2)}$.
Applying~\cite[Thm.~10.9]{Brezis_new},
we get that~\eqref{weaksol.weighted} possesses
a unique solution~$u$ satisfying
\begin{equation}\label{sol.smooth.weight.orginal}
  u\in L_{\mathrm{loc}}^2\big((0, \infty); \Dom(h_B^w)\big)
  \cap C^0\big([0,\infty);\sii_w(\Real^d)\big)
  \qquad \mbox{and} \qquad
  u_{,t}\in L_{\mathrm{loc}}^2\big((0, \infty); \Dom(h_B^w)^*\big)
  \,.
\end{equation}

\subsection{The self-similarity variables}
The difficulty in the study of the large-time behaviour
of the semigroup~\eqref{semigroup} is mainly due to
the lack of compactness of the resolvent of its generator~$H_B$.
To recover the compactness in the weighted Hilbert space~\eqref{weighted},
we apply the powerful method of self-similar variables,
which can be considered as a by now classical approach
to this type of problems (\cf~Introduction).

If $(x,t) \in \Real^d \times (0,\infty)$ are the initial	
space-time variables for the heat equation~\eqref{Cauchy},
we introduce the \emph{self-similar variables}
$(y,s) \in \Real^d \times (0,\infty)$ by
\begin{equation}\label{relationship}
  y := (t+1)^{-1/2} \, x
  \,, \qquad
  s := \log(t+1)
  \,.
\end{equation}
The angular variable~$\sigma$ of the spherical coordinates
is not changed by this transformation and for the radial one
we use the notation
\begin{equation}\label{relationship.spherical}
  \rho := |y| = (t+1)^{-1/2} \, |x| =  (t+1)^{-1/2} \, r
  \,.
\end{equation}

If~$u$ is a solution of~\eqref{Cauchy},
we then define a new function
\begin{equation}\label{SST}
  \tilde{u}(y, s) :=
  \Euler^{sd/4} \, u\big(\Euler^{s/2}y, \Euler^s-1\big)
  \,.
\end{equation}
The inverse transform is given by
\begin{equation}\label{ISST}
  u(x, t)=(t+1)^{-d/4} \, \tilde{u}\big((t+1)^{-1/2}x, \log(t+1)\big)
  \,.
\end{equation}
It is straightforward to check that~$\tilde{u}$
satisfies a weak formulation of the Cauchy problem
\begin{equation}\label{CauchySST}
\left\{
\begin{aligned}
  \tilde{u}_{,s}
  + \big(-i\nabla_{\!y} -A_s(y)\big)^2 \tilde{u}
  -\frac{c_d}{|y|^2}\tilde{u}
  -\frac{1}{2} \, y \cdot \nabla_{\!y} \tilde{u}
  -\frac{d}{4} \, \tilde{u}
  &=0
  \,,
  &(y, s)&\in \Real^d\times (0,\infty)
  \,,
  \\
  \tilde{u}(y, 0)&=u_0(y)
  \,,
  &y&\in \Real^d
  \,,
\end{aligned}
\right.
\end{equation}
with the new, $s$-dependent magnetic potential
\begin{equation}\label{scalevector}
  A_s(y) := \Euler^{s/2} A(\Euler^{s/2}y)
  \,.
\end{equation}
When evolution is posed in that context,
$y$~plays the role of the new space variable and~$s$ is the new time.
However, note that now we deal with a non-autonomous system
because of the presence of magnetic field.
\begin{Remark}
The same non-autonomous feature occurs and has been previously analysed
in the case of non-trivial geometries \cite{KZ1,KZ2,KKolb}
and also for a convection-diffusion equation in the whole space
but with a variable diffusion coefficient
\cite{Escobedo-Zuazua_1991,Duro-Zuazua_1999}.
A careful analysis of the behaviour of the underlying elliptic operators
as~$s$ tends to infinity leads to a sharp decay rate for its solutions.
\end{Remark}

To be more specific, the weak formulation of~\eqref{CauchySST}
is just the transformed version of~\eqref{weaksol} that reads
\begin{equation}\label{weak.SST}
  \tensor[_{\Dom(a_s)}]{\left\langle \phi,
  \tilde{u}_{,s}(s)
  - \frac{1}{2} \, y \cdot \nabla_{\!y} \tilde{u}(s)
  - \frac{d}{4} \tilde{u}(s)
  \right\rangle}{_{\Dom(a_s)^*}}
  +a_s\big(\phi, \tilde{u}(s)\big)=0
  \,,
\end{equation}
for each $\phi\in \Dom(a_s)$ and a.e.\ $s\in [0, \infty)$, and $u(0)=u_0$,
with the quadratic form
$$
\begin{aligned}
  a_s[\psi] :=\ &
  \big\|(\nabla-iA_s)\psi\big\|_{\sii(\Real^d)}^2
  - c_d \left\|\frac{\psi}{|y|}\right\|_{\sii(\Real^d)}^2
  \,, \qquad
  \Dom(a_s) := \ & \overline{C_0^\infty(\Real^d)}^{\|\cdot\|_{a_s}}
  \,.
\end{aligned}
$$
Here the norm~$\|\cdot\|_{a_s}$ is defined in analogy of~\eqref{norm}.
Recall that the form~$a_s$ is non-negative
due to~\eqref{diamagnetic} and~\eqref{Hardy}.

\subsection{Restricting the initial data to the weighted space}
The self-similarity transform $u \mapsto \tilde{u}$
acts as a unitary transform in $\sii(\Real^d)$;
indeed, we have
\begin{equation}\label{preserve}
  \|u(t)\|_{\sii(\Real^d)} = \|\tilde{u}(s)\|_{\sii(\Real^d)}
\end{equation}
for all $s,t \in (0,\infty)$.
This means that we can analyse the asymptotic time behaviour
of the former by studying the latter.

Because of the presence of the diffusion term in~\eqref{weak.SST}, however,
the natural space to study the evolution is not $\sii(\Real^d)$
but rather the weighted space $\sii_w(\Real^d)$
introduced in~\eqref{weighted}.
We thus define an additional transform
\begin{equation}\label{noweight}
  \tilde{v}(y,s) := w(y)^{1/2} \, \tilde{u}(y,s)
  \,,
\end{equation}
where the Gaussian weight~$w$ is defined in~\eqref{weighted}.
It casts~\eqref{CauchySST} \emph{formally} to
\begin{equation}\label{Cauchy.noweight}
\left\{
\begin{aligned}
  \tilde{v}_{,s}
  +  \big(-i\nabla_{\!y} -A_s(y)\big)^2 \tilde{v}
  -\frac{c_d}{|y|^2}\tilde{v}
  + \frac{|y|^2}{16} \tilde{v}
  - \frac{1}{2} i y \cdot A_s \, \tilde{v}
  &=0
  \,,
  &(y, s)&\in \Real^d\times (0,\infty)
  \,,
  \\
  \tilde{v}(y, 0)&= v_0(y)
  \,,
  &y&\in \Real^d
  \,,
\end{aligned}
\right.
\end{equation}
where $v_0 := w^{1/2} u_0$.
Hence, looking for solutions of~\eqref{CauchySST}
with an initial datum $u_0 \in \sii_w(\Real^d)$ is the same
as looking for solutions of~\eqref{Cauchy.noweight}
with the initial datum $v_0 \in \sii(\Real^d)$.
The advantage of the weighted setting is that
the presence of the harmonic-oscillator potential in~\eqref{Cauchy.noweight}
is responsible for the compactness of the resolvent
of the underlying elliptic operator.

Notice that the non-symmetric term on the first line of~\eqref{Cauchy.noweight}
vanishes provided that we choose~$A$ according to the Poincar\'e
gauge~\eqref{Poincare}.
From now on, we thus assume that~$A$ is given by~\eqref{Biot-Savart},
where the coefficients of the tensor~$B$ are smooth functions.

\subsection{Justifying the formal manipulations}
To show that~\eqref{Cauchy.noweight} is well posed,
we multiply the first line of~\eqref{Cauchy.noweight}
by an arbitrary test function $\phi \in C_0^\infty(\Real^d)$
and integrate over $y \in \Real^d$.
Then we \emph{formally} arrive at the identity
\begin{equation}\label{weakformweight}
  \tensor[_{\Dom(l_s)}]{\big\langle \phi, \tilde{v}_{,s}(s) \big\rangle}
  {_{\Dom(l_s)^*}}
  + l_s\big(\phi, \tilde{v}(s)\big)=0
\end{equation}
with the quadratic form
\begin{equation}\label{l-form}
\begin{aligned}
  l_s[\psi] :=
  \big\|(\nabla-iA_s)\psi\big\|_{\sii(\Real^d)}^2
  - c_d \left\| \frac{\psi}{|y|} \right\|_{\sii(\Real^d)}^2
  + \left\|\frac{|y|}{4}\psi\right\|_{\sii(\Real^d)}^2
  \!,
  \quad
  \Dom(l_s) := \ & \overline{C_0^\infty(\Real^d)}^{\|\cdot\|_{l_s}}
  . \!\!
\end{aligned}
\end{equation}
Here the norm with respect to which the closure is taken
is defined in analogy with~\eqref{norm}.
Note that~$l_s$ is non-negative due to
the diamagnetic inequality~\eqref{diamagnetic}
and the classical Hardy inequality~\eqref{Hardy}.

We also introduce the analogous form in the absence
of magnetic field (and thus $s$-in\-de\-pend\-ent)
\begin{equation}\label{l0-form}
\begin{aligned}
  l[\psi] :=
  \big\|\nabla\psi\big\|_{\sii(\Real^d)}^2
  - c_d \left\|\frac{\psi}{|y|}\right\|_{\sii(\Real^d)}^2
  + \left\|\frac{|y|}{4}\psi\right\|_{\sii(\Real^d)}^2
  \,,
  \qquad
  \Dom(l) := \ & \overline{C_0^\infty(\Real^d)}^{\|\cdot\|_{l}}
  \,.
\end{aligned}
\end{equation}
The following important result shows that the form domain~$\Dom(l_s)$
is locally independent of~$s$, provided that~$A$ is bounded.
\begin{Lemma}\label{Lem.norm.eq}
Let $d \geq 2$.
Suppose that~$A$ is bounded.
For any $s \geq 0$, there exists a positive constant~$C_s$
such that, for every $\psi \in C_0^\infty(\Real^d)$,
$$
  C_s^{-1} \|\psi\|_{l}  \leq \|\psi\|_{l_s} \leq C_s \|\psi\|_{l}
  \,.
$$
\end{Lemma}
\begin{proof}
As in Lemma~\ref{Lem.core}, it is possible to show that
the space $C_0^\infty(\Real^d \setminus\{0\})$ is a core
of both~$l_s$ and~$l$.
It is thus enough to prove the inequalities for
any fixed $\psi \in C_0^\infty(\Real^d \setminus\{0\})$.
Performing the substitution $g := \mathcal{V}\psi$,
where~$\mathcal{V}$ is given by~\eqref{unitary.Weidl},
we already know about the following identity
(\cf~\eqref{form.Weidl})
$$
  \big\|(\nabla-iA_s)\psi\big\|_{\sii(\Real^d)}^2
  - c_d \left\|\frac{\psi}{|y|}\right\|_{\sii(\Real^d)}^2
  = \int_{\Real^d}
  |(\nabla -iA_s) g(y)|^2 \, |y|^{-(d-2)}  \, \der y
  \,.
$$
Using the boundedness of~$A$ and elementary estimates,
we have
$$
  \delta |\nabla g|^2
  - \frac{\delta \, \Euler^s \|A\|_\infty^2}{1-\delta} |g|^2
  \leq |(\nabla -iA_s) g|^2 \leq
  2 |\nabla g|^2 + 2 \Euler^s \|A\|_\infty^2 |g|^2
$$
with any $\delta>0$.
Coming back to the original function~$\psi$
and choosing~$\delta$ small and $s$-dependent,
we establish the desired inequalities.
\end{proof}

As a consequence of~\eqref{A.vanish},
smooth~$A$ in the Poincar\'e gauge is bounded
under our characteristic assumption that~$B$ is compactly supported.
From now on, we thus assume that
the coefficients of~$B$ are compactly supported functions.

By ``formally'' above
we mean that it is not \emph{a priori} clear
that the solution~$\tilde{v}(s)$ and its derivative~$\tilde{v}_{,s}(s)$
belong to~$\Dom(l_s)$ and the dual~$\Dom(l_s)^*$, respectively,
so that the result~\eqref{weakformweight} of the formal manipulations
might be meaningless.
We therefore proceed conversely by showing first that~\eqref{Cauchy.noweight}
is actually well posed in $\sii(\Real^d)$.
\begin{Proposition}\label{Prop.Lions}
Let $d \geq 2$. Suppose that~$B$ is smooth, closed and compactly supported.
Choose~$A$ according to the gauge~\eqref{Biot-Savart}.
If $v_0 \in \sii(\Real^d)$, then the Cauchy problem~\eqref{Cauchy.noweight}
admits a unique weak solution
\begin{equation}\label{sol.smooth.weight}
  \tilde{v}\in L_{\mathrm{loc}}^2\big((0, \infty); \Dom(l)\big)
  \cap C^0\big([0,\infty);\sii(\Real^d)\big)
  \qquad \mbox{and} \qquad
  \tilde{v}_{,s}\in L_{\mathrm{loc}}^2\big((0, \infty); \Dom(l)^*\big)
  \,.
\end{equation}
More specifically,
there is such a unique solution of the problem
\begin{equation}\label{weak.form.weight}
  \tensor[_{\Dom(l)}]{\big\langle \phi, \tilde{v}_{,s}(s) \big\rangle}
  {_{\Dom(l)^*}}
  + l_s\big(\phi, \tilde{v}(s)\big)=0
\end{equation}
for each $\phi \in \Dom(l)$ and a.e.\ $s \in [0,\infty)$
satisfying $\tilde{v}(0)=v_0$.
\end{Proposition}
\begin{proof}
The existence of the weak solutions
follows by the theorem of J.~L.~Lions'~\cite[Thm.~10.9]{Brezis_new}
applied in the scale of Hilbert spaces
$
  \Dom(l) \subset \sii(\Real^d) \subset \Dom(l)^*
$
with help of Lemma~\ref{Lem.norm.eq}.
We refer to \cite[Prop.~5.1]{KZ1}
for the justification in an analogous situation
and leave the details to the reader.
\end{proof}

Now we are in a position to give a partial equivalence
of~\eqref{Cauchy} and~\eqref{Cauchy.noweight},
when the initial data of the former are restricted to
the weighted space~\eqref{weighted}.
\begin{Proposition}\label{Prop.justify}
Suppose the hypotheses of Proposition~\ref{Prop.Lions}.
If~$u$ satisfying~\eqref{sol.smooth.weight.orginal}
is a solution of~\eqref{weaksol.weighted} for each $\phi \in \Dom(h_B^w)$
and a.e.\ $t\in[0,\infty)$,
subject to the initial condition $u(0)=u_0 \in \sii_w(\Real^d)$,
then~$\tilde{v}$ defined in~\eqref{noweight} and~\eqref{SST}
is the solution of~\eqref{weak.form.weight} for each $\phi \in \Dom(a)$
and a.e.\ $s\in[0,\infty)$,
subject to the initial condition $\tilde{v}(0) = w^{1/2} u_0$,
and satisfies~\eqref{sol.smooth.weight}.
\end{Proposition}
\begin{proof}
It is straightforward to establish the identity
\begin{equation*}
  l_s[\tilde{v}] = \int_{\Real^d}
  \left[
  \Euler^s \left(|(\nabla-iA)u|^2 - c_d \frac{|u|^2}{|x|^2}\right)
  + \Euler^{-s} (1-\Euler^{s/2}) \frac{|x|^2}{8} |u(x)|^2
  - \frac{d}{4} |u|^2
  \right] w(\Euler^{-s/2}x) \, \der x
  \,,
\end{equation*}
provided that $u \in \Dom(h_B^w)$.
Noticing that $\Euler^{s/2} \geq 1$ and $w(\Euler^{-s/2}x) \leq w(x)$,
we obtain
$
  l_s[\tilde{v}] \leq \Euler^s \, h_B^w[u]
$,
whence $\tilde{v} \in \Dom(l_s)$ for every $s \geq 0$.
Then~\eqref{weakformweight} makes sense and holds.
Furthermore, by Proposition~\ref{Prop.Lions},
$\tilde{v}$~is the unique solution of~\eqref{weak.form.weight}
satisfying~\eqref{sol.smooth.weight}.
\end{proof}

As a consequence of this proposition, we may study the operator
$\Euler^{-t H_B} : \sii_w(\Real^d) \to \sii_w(\Real^d) \subset \sii(\Real^d)$
by analysing the evolution problem~\eqref{Cauchy.noweight}.

\subsection{Reduction to a spectral problem}
Choosing $\phi=\tilde{v}$ in~\eqref{weak.form.weight}
and combining the obtained equation with its conjugate version,
we arrive at the identity
\begin{equation}\label{energy}
  \frac{1}{2} \frac{\der}{\der s} \|\tilde{v}(s)\|_{\sii(\Real^d)}^2
  = - l_s[\tilde{v}(s)]
\end{equation}
for every $s \geq 0$.
Now, as usual for energy estimates,
we replace the right hand side of~\eqref{energy}
by the spectral bound
\begin{equation}\label{spec.bound}
  l_s[\tilde{v}(s)] \geq \lambda_B(s) \, \|\tilde{v}(s)\|_{\sii(\Real^d)}^2
  \,,
\end{equation}
where~$\lambda_B(s)$ is the lowest point in the spectrum of
the self-adjoint operator~$L_s$ in $\sii(\Real^d)$ associated with the form~$l_s$.
Then~\eqref{energy} together with~\eqref{spec.bound} implies Gronwall's inequality
\begin{equation}\label{Gronwall}
  \|\tilde{v}(s)\|_{\sii(\Real^d)}
  \leq \|v_0\|_{\sii(\Real^d)} \
  \Euler^{-\int_0^s \lambda_B(\tau) \, \der\tau}
\end{equation}
valid for every $s \geq 0$.
Recall that~$L_s$ is non-negative for every $s \geq 0$.

In this way, the problem of large-time behaviour of~\eqref{semigroup}
is reduced to a spectral analysis
of the family of operators $\{L_s\}_{s \geq 0}$.
In particular, we have the following results.
\begin{Proposition}\label{Prop.reduction}
Suppose the hypotheses of Proposition~\ref{Prop.Lions}.
Then
\begin{enumerate}
\item[\emph{(i)}]
$
  \big\|\Euler^{-t H_B}\big\|_{\sii_w(\Real^d)\to\sii(\Real^d)}
  \leq
  (1+t)^{-\lambda_B^\mathrm{min}}
$
for every $t \geq 0$,
where $\displaystyle \lambda_B^\mathrm{min} := \inf_{s \geq 0} \lambda_B(s)$.
\item[\emph{(ii)}]
$\gamma_B \geq \lambda_B(\infty)$,
where $\displaystyle \lambda_B(\infty) := \liminf_{s \to \infty} \lambda_B(s)$.
\end{enumerate}
\end{Proposition}
\begin{proof}
Using~\eqref{preserve}, the pointwise bound $w \geq 1$,
inequality~\eqref{Gronwall}
and the identification $v_0 = w^{1/2} u_0$,
we have
$$
  \|u(t)\|_{\sii(\Real^d)}
  = \|\tilde{u}(s)\|_{\sii(\Real^d)}
  \leq \|\tilde{u}(s)\|_{\sii_w(\Real^d)}
  = \|\tilde{v}(s)\|_{\sii(\Real^d)}
  \leq \|u_0\|_{\sii_w(\Real^d)} \
  \Euler^{-\int_0^s \lambda_B(\tau) \, \der\tau}
  \,.
$$
Consequently,
$$
  \big\|\Euler^{-t H_B}\big\|_{\sii_w(\Real^d)\to\sii(\Real^d)}
  = \sup_{u_0 \in \sii_w(\Real^d) \setminus \{0\}}
  \frac{\|u(t)\|_{\sii(\Real^d)}}{\|u_0\|_{\sii_w(\Real^d)}}
  \leq
  \Euler^{-\int_0^s \lambda_B(\tau) \, \der\tau}
  \,.
$$
Then~(i) immediately follows from the crude bound
$
  \Euler^{-\int_0^s \lambda_B(\tau) \, \der\tau}
  \leq \Euler^{-\lambda_B^\mathrm{min} s}
  = (1+t)^{-\lambda_B^\mathrm{min}}
$,
where the equality is due to the relationship~\eqref{relationship}
between~$s$ and~$t$.
To prove~(ii), we refer to analogous situations in
\cite[Sec.~5.8]{KZ1}, \cite[Sec.~4.5]{KZ2}, \cite[Sec.~3.5]{K7}
or \cite[Sec.~7.10]{KKolb}.
\end{proof}

The behaviour of~$\lambda_B(s)$ on~$s$
will be studied in the following section.
Using additional results about~$L_s$,
we shall actually show that there holds
an equality in Proposition~\ref{Prop.reduction}.(ii).

\section{Schr\"odinger operators with singularly scaled magnetic field}
\label{Sec.singular}
%
In this section we eventually give a proof of Theorem~\ref{Thm.main}
by analysing the family of operators $\{L_s\}_{s \geq 0}$.
Recall that, for each fixed $s \geq 0$,
$L_s$~is the self-adjoint operator in $\sii(\Real^d)$
associated with the sesquilinear form~\eqref{l-form}.
On its domain the operator acts as
\begin{equation}\label{Schrodinger.scaled}
  L_s = \big(-i\nabla_{\!y}-A_s(y)\big)^2 - \frac{c_d}{|y|^2}
  + \frac{|y|^2}{16}
  \,,
\end{equation}
where~$A_s$ is the singularly scaled magnetic potential~\eqref{scalevector}.
We shall be particularly interested in the asymptotic behaviour
of~$L_s$ as $s \to \infty$.
Throughout this section, we assume that~$A$ is smooth and bounded.

\subsection{Basic properties}
First of all, we remark that, because of the presence
of the unbounded harmonic-oscillator potential in~\eqref{Schrodinger.scaled},
we have the following important result.
\begin{Proposition}\label{Prop.compact}
Let $d \geq 2$. Suppose that~$A$ is bounded.
Then~$L_s$ is an operator with compact resolvent for any $s \geq 0$.
\end{Proposition}
\begin{proof}
In view of Lemma~\ref{Lem.norm.eq}, it is enough to show
that $\Dom(l)$ is compactly embedded in $\sii(\Real^d)$.
This property can be established by standard methods (\cf~\cite[Sec.~XIII.14]{RS4}),
but one can alternatively recall an existing result from~\cite{Vazquez-Zuazua_2000}:
For any $\psi \in C_0^\infty(\Real^d)$, it is easy to see that
$$
  l[w^{1/2}\psi] = \|\nabla\psi\|_{\sii_w(\Real^d)}^2
  - c_d \left\|\frac{\psi}{|y|}\right\|_{\sii_w(\Real^d)}^2
  - \frac{d}{4} \, \|\psi\|_{\sii_w(\Real^d)}^2
$$
and the completion of $C_0^\infty(\Real^d)$
with respect to the norm induced by the quadratic form
on the right hand side is compactly embedded in $\sii_w(\Real^d)$
due to \cite[Prop.~9.2]{Vazquez-Zuazua_2000}.
\end{proof}

Consequently, $L_s$~has a purely discrete spectrum.
In particular, $\lambda_B(s)$~represents an eigenvalue of~$L_s$ for any $s \geq 0$.
By the variational characterisation of the spectrum of~$L_s$,
we have
\begin{equation}\label{lambdaB}
  \lambda_B(s) = \min_{\stackrel[\psi \not= 0 ]{}{\psi \in \Dom(l)}}
  \frac{l_s[\psi]}{\|\psi\|_{\sii(\Real^d)}^2}
  \,,
\end{equation}
where we can indeed take $\Dom(l)$ instead of $\Dom(l_s)$
due to Lemma~\ref{Lem.norm.eq}.

\subsection{No magnetic field}
In the absence of magnetic field, \ie~$B=0$,
we can always choose $A=0$.
In this case, $L_s$~is independent of~$s$
and it coincides with the operator~$L$
associated with the quadratic form~$l$ introduced in~\eqref{l0-form}.
The spectral problem for~$L$ can be solved explicitly
by means of a separation of variables (\cf~\cite[Prop.~3]{K7}).
\begin{Proposition}\label{Prop.Laguerre}
Let $d \geq 2$. We have
$$
  \sigma(L) =
  \left\{n+\frac{1+\sqrt{\ell(\ell+d-2)}}{2}\right\}_{n,\ell \in \Nat}
  \,.
$$
\end{Proposition}

We point out that the natural numbers~$\Nat$ contain~$0$ in our convention.
In particular, for the lowest eigenvalue~\eqref{lambdaB}, we get
\begin{Corollary}\label{Corol.1/2}
Let $d \geq 2$.
We have $\lambda_0(s) = 1/2$ for all $s \geq 0$.
\end{Corollary}

It follows from the corollary that $\lambda_0(\infty)=1/2$
and therefore $\gamma_0 \geq 1/2$ due to Proposition~\ref{Prop.reduction}.
This together with the following result
proves Theorems~\ref{Thm.main.pre} and~\ref{Thm.main}
in the magnetic-free case $B=0$.

\begin{Proposition}\label{Prop.optimal}
Let $d \geq 2$.
We have $\gamma_0 = 1/2$.
\end{Proposition}
\begin{proof}
We have already established $\gamma_0 \geq 1/2$.
To prove the opposite bound, it is enough to find
an initial datum $u_0 \in \sii_w(\Real^d)$
such that the solution of~\eqref{Cauchy} satisfies the inequality
$\|u(t)\|_{\sii(\Real^d)} \geq c \, (1+t)^{-1/2}$
for all $t \geq 0$ with some positive constant~$c$
that may depend on~$u_0$.
Since $\psi_1(y):=|y|^{-(d-2)/2} \Euler^{-|y|^2/8}$
is an eigenfunction of~$L$ corresponding to the eigenvalue~$1/2$,
the function $\tilde{v}(y,s):=\Euler^{-s/2} \psi_1(y)$
solves~\eqref{Cauchy.noweight},
subject to the initial condition $\tilde{v}(y,0)=\psi_1(y)$.
Defining~$u$ by means of~\eqref{ISST} and~\eqref{noweight},
we get that it solves~\eqref{Cauchy},
subject to the initial condition $u(x,0)=w(x)^{-1/2}\psi_1(x)=:u_0(x)$.
Clearly, $u_0 \in \sii_w(\Real^d)$.
Recalling~\eqref{preserve} and~\eqref{noweight} again, we get
$$
  \|u(t)\|_{\sii(\Real^d)}
  = \|w^{-1/2}\tilde{v}(s)\|_{\sii(\Real^d)}
  = \Euler^{-s/2} \, \|u_0\|_{\sii(\Real^d)}
  = (1+t)^{-1/2} \, \|u_0\|_{\sii(\Real^d)}
$$
for all $t \geq 0$,
where the last identity follows by the relationship~\eqref{relationship}.
\end{proof}

\subsection{Lower bounds}
In this subsection, we focus on part~(i) of Proposition~\ref{Prop.reduction}.
First of all, using just the diamagnetic inequality~\eqref{diamagnetic}
and Corollary~\ref{Corol.1/2}, we immediately get the following results.
\begin{Proposition}\label{Prop.bound}
Let $d \geq 2$. If~$B$ is smooth and closed, then
$\lambda_B(s) \geq 1/2$ for all $s \geq 0$.
\end{Proposition}
\begin{Corollary}
Let $d \geq 2$.
If~$B$ is smooth, closed and compactly supported,
then $\gamma_B \geq 1/2$.
\end{Corollary}

From the previous subsection, we already know that both bounds
are optimal in the absence of magnetic field.
Now we show that there is always an improvement
in the case of Proposition~\ref{Prop.bound}
whenever~$B$ is non-trivial.
\begin{Proposition}\label{Prop.strict}
Let $d \geq 2$. Suppose that~$B$ is smooth and closed.
If $B \not=0 $, then
$$
  \forall s \geq 0 \,, \qquad
  \lambda_B(s) > \frac{1}{2}
  \,.
$$
\end{Proposition}
\begin{proof}
By Proposition~\ref{Prop.bound}, we already know that
$\lambda_B(s) \geq 1/2$ for all $s \geq 0$.
In order to show that the inequality is strict,
we assume by contradiction that $B \not= 0$
and $\lambda_B(s)=1/2$ for some $s \geq 0$.
Let~$\psi$ denote a corresponding eigenfunction of~$L_s$. 
By elliptic regularity theory, we know that~$\psi$
is smooth in $\Real^d \setminus \{0\}$.
We have 
 \begin{equation}\label{strict}
  0 = l_s[\psi] - \frac{1}{2} \|\psi\|_{\sii(\Real^d)}^2
  \geq l[|\psi|] - \frac{1}{2} \|\psi\|_{\sii(\Real^d)}^2
  \geq 0
  \,, 
\end{equation}
where the first estimate is the diamagnetic inequality~\eqref{diamagnetic}
and the second estimate follows from the variational characterisation
of the first eigenvalue $\lambda_0=1/2$ of~$L$.
Note that the eigenfunction of~$L$ corresponding to $\lambda_0=1/2$
is unique (up to a normalisation factor) and can be chosen positive.
It then follows from~\eqref{strict} that
 the magnitude~$|\psi|$
must coincide with the first eigenfunction of~$L$.
Writing $\psi = |\psi| \Euler^{i\varphi}$,
where~$\varphi$ is a real-valued smooth function in $\Real^d \setminus \{0\}$,
we thus obtain from~\eqref{strict} that $\nabla\varphi = A_s$
in $\Real^d \setminus \{0\}$.
That is, $A_s$~is exact and thus $\der A_s=0$
in the punctured space $\Real^d \setminus \{0\}$.
From this we conclude that $B=\der A=0$ in~$\Real^d$,
a contradiction.
\end{proof}

In order to apply Proposition~\ref{Prop.reduction}.(i),
we also need to ensure a strict positivity of $\lambda_B(\infty)-1/2$.
The following result shows that this is not always possible.
\begin{Proposition}\label{Prop.log}
Let $d \geq 2$. Suppose that~$B$ is smooth, closed and compactly supported.
Then
$$
  \nu_B(\infty)=0
  \qquad \Longrightarrow \qquad
  \lambda_B(\infty)=\frac{1}{2}
  \,.
$$
\end{Proposition}
\begin{proof}
In view of Proposition~\ref{Prop.bound}, it is enough to establish
an upper bound to~$\lambda_B(s)$ that goes to~$1/2$ as $s\to\infty$.
Following \cite[Prop.~2]{K7},
we obtain it by constructing a suitable trial function
in the variational characterisation of the eigenvalue.

Inspired by \cite[proof of Corol.~VIII.6.4]{Edmunds-Evans},
let~$\xi$ be a smooth real-valued function on~$[0,1]$ such that $\xi=0$
in a right neighbourhood of~$0$ and $\xi=1$ in a left neighbourhood of~$1$.
For any natural number $n \geq 2$,
we define a smooth cut-off function $\eta_n:[0,\infty)\to[0,1]$ by
$$
\displaystyle
\eta_n(\rho) :=
\begin{cases}
  0
  & \mbox{if}\quad \rho<1/n^2 \,,
  \\
  \xi\big(\log_n(n^2 \rho)\big)
  & \mbox{if}\quad \rho\in[1/n^2,1/n] \,,
  \\
  1
  & \mbox{if}\quad \rho>1/n \,.
\end{cases}
$$
We have the limits $\eta_n(\rho) \to 1$ as $n \to \infty$ for every $\rho>0$ and
\begin{equation}\label{eta.convergence}
  \|\eta_n'\|_{\sii((0,\infty),\rho \der \rho)}^2
  \leq \frac{\|\xi'\|_\infty^2}{\log n}
  \xrightarrow[n\to\infty]{}
  0
  \,.
\end{equation}
Hence, the functions $y \mapsto \eta_n(|y|)$ represent
a convenient smooth approximation of the constant function~$1$
in $W_\mathrm{loc}^{1,2}(\Real^2)$
when a singularity in the origin has to be avoided
(\cf~Lemma~\ref{Lem.core}).

We work in the spherical coordinates~\eqref{spherical}
and the Poincar\'e gauge~\eqref{Biot-Savart}.
We define
$$
  g_n(y) := \psi_1(|y|) \, \eta_n(|y|) \, \varphi(y/|y|)
  \,,
$$
where $\psi_1(\rho) := \rho^{-(d-2)/2} \, \Euler^{-\rho^2/8}$
is the eigenfunction of~$L$
corresponding to its first eigenvalue $\lambda_0=1/2$
and~$\varphi$ is a non-trivial smooth solution of
$\der'\varphi-i\mathsf{A}_\infty\varphi=0$ on~$S^{d-1}$.
In view of Proposition~\ref{Prop.eq.r},
such a solution exists due to the hypothesis $\nu_B(\infty)=0$.
We clearly have $g_n \in \Dom(l_s)$ for all $n \geq 2$ and $s \geq 0$.
Passing to the spherical coordinates,
one easily checks the identity (\cf~\eqref{form.spherical})
\begin{align}\label{Lebesgue}
  l_s[g_n] - \frac{1}{2} \, \|g_n\|_{\sii(\Real^d)}^2
  \ = \ &
  \int_{S^{d-1}\times(0,\infty)}
  |\psi_1(\rho)|^2 \, |\eta_n(\rho)|^2 \,
  \frac{\big|
  [\mathsf{A}_\infty(\sigma) - \mathsf{A}_s(\sigma,\rho)]\varphi(\sigma)
  \big|_{S^{d-1}}^2}{\rho^2}
  \, \rho^{d-1} \;\! \der\sigma \, \der \rho
  \nonumber \\
  & + \int_{S^{d-1}\times(0,\infty)}
  |\psi_1(\rho)|^2 \, |\eta_n'(\rho)|^2 \, |\varphi(\sigma)|^2
  \, \rho^{d-1} \;\! \der\sigma \, \der \rho
  \,,
\end{align}
where $\mathsf{A}_s := \nabla\chart \cdot (A_s\circ\chart)$.
Using~\eqref{transfer} and~\eqref{relationship}, we find
\begin{equation}\label{vectorscaled.spherical}
  \mathsf{A}_s(\sigma,\rho)
  = \mathsf{A}(\sigma,\Euler^{s/2} \rho)
  \,.
\end{equation}
Consequently, the first integral on the right hand side of~\eqref{Lebesgue}
goes to zero as $s \to \infty$ by the dominated convergence theorem.
The second term is independent of~$s$ so as is $\|g_n\|_{\sii(\Real^d)}$.
From the variational characterisation~\eqref{lambdaB} we thus obtain
\begin{align*}
  0 \leq \lambda_B(\infty) - \frac{1}{2}
  &\leq \frac{\displaystyle \int_{S^{d-1}\times(0,\infty)}
  |\psi_1(\rho)|^2 \, |\eta_n'(\rho)|^2 \, |\varphi(\sigma)|^2
  \, \rho^{d-1} \;\! \der\sigma \, \der \rho}
  {\displaystyle \int_{S^{d-1}\times(0,\infty)}
  |\psi_1(\rho)|^2 \, |\eta_n(\rho)|^2 \, |\varphi(\sigma)|^2
  \, \rho^{d-1} \;\! \der\sigma \, \der \rho}
  = \frac{\displaystyle \int_{0}^\infty
  \Euler^{-\rho^2/4} \, |\eta_n'(\rho)|^2
  \, \rho \, \der \rho}
  {\displaystyle \int_{0}^\infty
  \Euler^{-\rho^2/4} \, |\eta_n(\rho)|^2
  \, \rho \, \der \rho}
\end{align*}
for every $n \geq 2$.
In view of~\eqref{eta.convergence} and the text before it,
the right hand goes to zero in the limit $n \to \infty$.
We have thus established the desired upper bound $\lambda_B(\infty) \leq 1/2$.
\end{proof}

It remains to study the asymptotic behaviour of $\lambda_B(s)$ as $s \to \infty$
in the case $\nu_B(\infty)\not=0$.
From now on, we could restrict to $d=2$,
since $\nu_B(\infty)=0$ whenever $d \geq 3$, \cf~\eqref{d.high}.
However, there is no complication in continuing with the general setting,
noting that the statements in the higher dimensions will be just void.

\subsection{The asymptotic behaviour}
It will be convenient to work in the spherical coordinates~\eqref{spherical}
and the Poincar\'e gauge~\eqref{Biot-Savart}.

Recalling~\eqref{unitary.spherical}, we may write
\begin{equation}\label{ls.spherical}
  l_s[\mathcal{U}^{-1}\phi] =
  \int_{S^{d-1}\times(0,\infty)}
  \left[
  \frac{\big|(\der' - i\mathsf{A}_s)\phi\big|_{S^{d-1}}^2}{\rho^2}
  + |\phi_{,\rho}|^2
  - c_d \, \frac{|\phi|^2}{\rho^2}
  + \frac{\rho^2}{16} \, |\phi|^2
  \right]
  \rho^{d-1} \;\! \der\sigma \, \der \rho
\end{equation}
for any $\phi \in C_0^\infty\big(S^{d-1}\times(0,\infty)\big)$,
where~$\mathsf{A}_s$ is given by~\eqref{vectorscaled.spherical}.
The unitarily equivalent operator $\mathsf{L}_s := \mathcal{U}L_s\mathcal{U}^{-1}$
in the Hilbert space
$
  \mathcal{H} :=
  \sii\big(S^{d-1}\times(0,\infty),\rho^{d-1} \;\! \der\sigma \, \der \rho\big)
$
is thus associated with the quadratic form
$\mathsf{l}_s[\phi] := l_s[\mathcal{U}^{-1}\phi]$,
$\Dom(\mathsf{l}_s) := \mathcal{U} \Dom(l_s)$.
Of course, we have
$$
  \Dom(\mathsf{l}_s) =
  \overline{C_0^\infty\big(S^{d-1}\times(0,\infty)\big)}^{\|\cdot\|_{\mathsf{l}_s}}
  \,,
$$
where the norm~$\|\cdot\|_{\mathsf{l}_s}$ is defined in analogy with~\eqref{norm}.
By Lemma~\ref{Lem.norm.eq}, we know that $\Dom(\mathsf{l}_s)$
is actually independent of~$s$.
In fact, for any finite $s \geq 0$,
$$
  \Dom(\mathsf{l}_s) = \Dom(\mathsf{l}) :=
  \overline{C_0^\infty\big(S^{d-1}\times(0,\infty)\big)}^{\|\cdot\|_{\mathsf{l}}}
  \,,
$$
where (\cf~\eqref{l0-form})
$$
  \mathsf{l}[\phi] :=
  \int_{S^{d-1}\times(0,\infty)}
  \left[
  \frac{|\der'\phi|_{S^{d-1}}^2}{\rho^2}
  + |\phi_{,\rho}|^2
  - c_d \, \frac{|\phi|^2}{\rho^2}
  + \frac{\rho^2}{16} \, |\phi|^2
  \right]
  \rho^{d-1} \;\! \der\sigma \, \der \rho
  \,.
$$

By~\eqref{Hardy.1D},
we see that the expression on the right hand side of the last formula
is composed of three non-negative terms.
More specifically, recalling~\eqref{Hardy.transform}, we have
\begin{equation*}
  \int_{S^{d-1}\times (0, \infty)} \left[|\phi_{, \rho}|^2
  -c_d \, \frac{|\phi|^2}{\rho^2} \right] \rho^{d-1} \, \der  \sigma \, \der \rho
  = \int_{S^{d-1}\times (0, \infty)}
  \left|
  \rho^{-(d-2)/2}
  \left(\rho^{(d-2)/2} \phi\right)_{\!, \rho}
  \right|^2
  \rho^{d-1} \, \der \sigma \, \der \rho
\end{equation*}
and we may thus conclude
\begin{equation}\label{Doml}
  \Dom(\mathsf{l}) = \left\{
  \phi \in \mathcal{H} \ \left| \
  \frac{|\der'\phi|_{S^{d-1}}}{\rho}, \
  \rho^{-(d-2)/2} \left(\rho^{(d-2)/2} \phi\right)_{\!, \rho} , \
  \frac{|\phi|^2}{\rho^2}, \
  \rho \phi \in \mathcal{H}
  \right.
  \right\}.
\end{equation}

Taking into account~\eqref{vectorscaled.spherical} and~\eqref{limits},
it is reasonable to expect that the behaviour of~$\mathsf{L}_s$
in the limit $s\to\infty$ will be determined by the operator~$\mathsf{L}_\infty$
associated with the the quadratic form
$$
\begin{aligned}
  \mathsf{l}_\infty[\phi] &:=
  \int_{S^{d-1}\times(0,\infty)}
  \left[
  \frac{\big|(\der' - i\mathsf{A}_\infty)\phi\big|_{S^{d-1}}^2}{\rho^2}
  + |\phi_{,\rho}|^2
  - c_d \, \frac{|\phi|^2}{\rho^2}
  + \frac{\rho^2}{16} \, |\phi|^2
  \right]
  \rho^{d-1} \;\! \der\sigma \, \der \rho
  \,,
  \\
  \Dom(\mathsf{l}_\infty) &:=
  \overline{C_0^\infty\big(S^{d-1}\times(0,\infty)\big)}^{\|\cdot\|_{\mathsf{l}_\infty}}
  \,.
\end{aligned}
$$
There is a substantial difference between~$\mathsf{l}_s$ and~$\mathsf{l}_\infty$,
because~$\mathsf{A}_\infty$ is singular in the sense that
it is independent of the radial variable.
Consequently, $\Dom(\mathsf{l}_\infty)$ is smaller.
Indeed, recalling~\eqref{nu.r}, we obtain a Hardy-type inequality
$$
  \forall \phi \in C_0^\infty\big(S^{d-1}\times(0,\infty)\big)
  \,, \qquad
  \mathsf{l}_\infty[\phi] \geq
  \nu_B(\infty)
  \int_{S^{d-1}\times(0,\infty)}
  \frac{|\phi|^2}{\rho^2} \,
  \rho^{d-1} \;\! \der\sigma \, \der \rho
  \,,
$$
from which it follows that
\begin{equation}\label{implication}
  \nu_B(\infty) \not= 0
  \qquad \Longrightarrow \qquad
  \Dom(\mathsf{l}_\infty) =
  \left\{
  \phi \in \Dom(\mathsf{l}) \ \left| \
  \frac{\phi}{r} \in \mathcal{H}
  \right.
  \right\}
  \,.
\end{equation}
On the other hand, since the vector potential~$\mathsf{A}_\infty$
can be gauged out if $\nu_B(\infty)=0$, \cf~Proposition~\ref{Prop.eq.r},
we have $\Dom(\mathsf{l}_\infty) = \Dom(\mathsf{l})$ in this case.
In any case, recalling Proposition~\ref{Prop.compact},
we can easily deduce
that~$\mathsf{L}_\infty$ is an operator with compact resolvent.

The following theorem is probably the most important
auxiliary result of this paper.
\begin{Theorem}\label{Prop.strong}
Let $d \geq 2$. Suppose that~$B$ is smooth, closed and compactly supported.
If $\nu_B(\infty) \not= 0$, then the operator~$\mathsf{L}_s$ converges
to~$\mathsf{L}_\infty$ in the norm-resolvent sense as $s \to \infty$, \ie,
\begin{equation}\label{nrs}
  \lim_{s \to \infty}
  \big\|\mathsf{L}_s^{-1}-\mathsf{L}_\infty^{-1}\big\|_{\mathcal{H}\to\mathcal{H}}
  = 0
  \,.
\end{equation}
\end{Theorem}
\begin{proof}
First of all, we notice that it follows from Proposition~\ref{Prop.bound}
that~$0$ belongs to the resolvent set of~$\mathsf{L}_s$ for all $s \geq 0$
and the same argument applies to~$\mathsf{L}_\infty$.
To prove the uniform convergence~\eqref{nrs},
we shall use an abstract criterion of Lemma~\ref{Lem.nrs} from the appendix.
Let $\{s_n\}_{n \in \Nat}$ be an arbitrary
sequence of non-negative integers such that $s_n \to \infty$ as $n \to \infty$
and let $\{f_n\}_{n \in \Nat} \subset \mathcal{H}$
be an arbitrary family of functions
weakly converging to~$f$ and such that
$\|f_n\|_{\mathcal{H}}=1$ for all $n \in \Nat$.
By Lemma~\ref{Lem.nrs}, it is enough to show that
\begin{equation}\label{want}
  \lim_{n \to \infty}
  \big\|\mathsf{L}_{s_n}^{-1} f_n-\mathsf{L}_\infty^{-1}f\big\|_{\mathcal{H}}
  = 0
  \,.
\end{equation}

We set $\phi_n := \mathsf{L}_{s_n}^{-1} f_n$, so that~$\phi_n$
satisfies the weak formulation of the resolvent equation
\begin{equation}\label{re}
  \forall v \in \Dom(\mathsf{l}) \,, \qquad
  \mathsf{l}_{s_n}(v,\phi_n)
  = (v,f_n)_{\mathcal{H}}
  \,.
\end{equation}
Choosing $v = \phi_n$ for the test function in~\eqref{re},
we have
\begin{equation}\label{resolvent.identity}
  \mathsf{l}_{s_n}[\phi_n]
  = (\phi_n,f_n)_{\mathcal{H}}
  \leq \|\phi_n\|_{\mathcal{H}} \|f_n\|_{\mathcal{H}}
  = \|\phi_n\|_{\mathcal{H}}
  \,.
\end{equation}

Noticing that Proposition~\ref{Prop.bound} yields
the Poincar\'e-type inequality
$\mathsf{l}_{s}[\phi] \geq \frac{1}{2} \|\phi\|_\mathcal{H}^2$
for any $\phi \in \Dom(\mathsf{l})$,
we obtain from~\eqref{resolvent.identity}
the uniform bound
\begin{equation}\label{bound0}
  \|\phi_n\|_{\mathcal{H}} \leq 2
  \,.
\end{equation}
At the same time, recalling~\eqref{ls.spherical},
the bounds~\eqref{resolvent.identity} and~\eqref{bound0} yield
\begin{equation}\label{bounds}
  \left\|
  \frac{\big|(\der' - i\mathsf{A}_{s_n})\phi_n\big|_{S^{d-1}}}{\rho}
  \right\|_{\mathcal{H}}^2
  \leq 2
  \,, \qquad
  \left\|
  \rho^{-(d-2)/2} \left(\rho^{(d-2)/2} \phi_n\right)_{\!, \rho}
  \right\|_{\mathcal{H}}^2
  \leq 2
  \,, \qquad
  \|\rho\phi_n\|_{\mathcal{H}}^2 \leq 32
  \,.
\end{equation}

If $B=0$, then $\mathsf{A}=0$, and consequently,
the first inequality of~\eqref{bounds}
is already a uniform bound on the angular derivative of~$\phi_n$,
so that~\eqref{bound0} and~\eqref{bounds} imply that $\{\phi_n\}_{n \in \Nat}$
is a bounded family in~$\Dom(\mathsf{l})$, \cf~\eqref{Doml}.
To get a similar estimate on the angular derivative of~$\phi_n$
in our situation $\nu_B(\infty) \not= 0$,
we employ the presence of magnetic Hardy inequalities as follows.
In the simultaneous usage of the spherical coordinates~\eqref{spherical},
self-similarity variables \eqref{relationship}--\eqref{relationship.spherical}
and the Poincar\'e gauge~\eqref{Biot-Savart},
the Hardy inequality~\eqref{Hardy.magnet.LW} of Theorem~\ref{Thm.Hardy.LW}
reads
\begin{equation*}
  \left\|
  \frac{\big|(\der' - i\mathsf{A}_{s})\phi\big|_{S^{d-1}}}{\rho}
  \right\|_{\mathcal{H}}^2
  + \left\|
  \rho^{-(d-2)/2} \left(\rho^{(d-2)/2} \phi\right)_{\!, \rho}
  \right\|_{\mathcal{H}}^2
  \geq \tilde{c}_{d,B} 	
  \int_{S^{d-1}\times(0,\infty)} \frac{|\phi|^2}
  {\Euler^{-{s}}+ \rho^2} \,
  \rho^{d-1} \;\! \der\sigma \, \der \rho
\end{equation*}
for any $\phi \in C_0^\infty\big(S^{d-1}\times(0,\infty)\big)$.
Applying this inequality to~\eqref{resolvent.identity}
and recalling~\eqref{bound0}, we therefore get
\begin{equation}\label{ri3}
  \int_{S^{d-1}\times(0,\infty)} \frac{|\phi_n|^2}
  {\Euler^{-s_n}+ \rho^2} \,
  \rho^{d-1} \;\! \der\sigma \, \der \rho
  \leq \frac{2}{\tilde{c}_{d,B}}
  \,.
\end{equation}
Consequently,
\begin{equation}\label{C1}
  \left\|
  \frac{|\mathsf{A}_{s_n}\phi_n|_{S^{d-1}}}{\rho}
  \right\|_{\mathcal{H}}^2
  \leq \frac{2}{\tilde{c}_{d,B}} \, C
  \,,
\end{equation}
where the number
$$
  C := \sup_{(\sigma,r) \in S^{d-1} \times (0,\infty)}
  \frac{1 + r^2}{r^2} \,
  |\mathsf{A}(\sigma,r)|_{S^{d-1}}^2
$$
is indeed a finite constant because
$|\mathsf{A}(\sigma,r)|_{S^{d-1}} = \mathcal{O}(r)$ as $r \to 0$
as a consequence of~\eqref{passing}
and $\mathsf{A}(\sigma,r)$
equals the smooth vector field $\mathsf{A}_\infty(\sigma)$
for all sufficiently large~$r$.
Using~\eqref{C1} in the first term of~\eqref{bounds},
we eventually get the desired estimate
\begin{equation}\label{bound.angular}
  \left\|
  \frac{|\der'\phi_n|_{S^{d-1}}}{\rho}
  \right\|_{\mathcal{H}}^2
  \leq \tilde{C}
  \,,
\end{equation}
where~$\tilde{C}$ is a constant independent of~$n$.

It follows from~\eqref{bound0}, \eqref{bounds} and~\eqref{bound.angular}
that $\{\phi_n\}_{n \in \Nat}$ is a bounded family in $\Dom(\mathsf{l})$
equipped with the norm $\|\cdot\|_{\mathsf{l}}$.
Therefore it is precompact in the weak topology of this space.
Let~$\phi_\infty$ be a weak limit point,
\ie, for a sequence $\{n_j\}_{j\in\Nat}$ of non-negative integers
such that $n_j \to \infty$ as $j \to \infty$,
$\{\phi_{n_j}\}_{j\in\Nat}$ converges weakly to~$\phi_\infty$ in~$\Dom(\mathsf{l})$.
Actually, we may assume that the sequence converges strongly in~$\mathcal{H}$
because~$\Dom(\mathsf{l})$ is compactly embedded in~$\mathcal{H}$.
Summing up,
\begin{equation}\label{weak1}
  \phi_{n_j} \xrightarrow[j \to \infty]{w} \phi_\infty
  \quad \mbox{in} \quad \Dom(\mathsf{l})
  \qquad \mbox{and} \qquad
  \phi_{n_j} \xrightarrow[j \to \infty]{} \phi_\infty
  \quad \mbox{in} \quad \mathcal{H}
  \,.
\end{equation}

Our next objective is to show that $\phi_\infty \in \Dom(\mathsf{l}_\infty)$.
Since $\{\phi_{n_j}\}_{j\in\Nat}$ converges strongly to~$\phi_\infty$ in~$\mathcal{H}$,
we clearly have
$$
  \forall \phi \in C_0^\infty\big(S^{d-1}\times(0,\infty)\big)
  \,, \qquad
  \left(\phi,
  \frac{\phi_{n_j}}{\sqrt{\Euler^{-{s_{n_j}}}+\rho^2}}\right)_{\!\mathcal{H}}
  \xrightarrow[j\to\infty]{}
  \left(\phi,
  \frac{\phi_\infty}{\rho}\right)_{\!\mathcal{H}}
  \,.
$$
Since $C_0^\infty\big(S^{d-1}\times(0,\infty)\big)$ is dense in~$\mathcal{H}$
and the uniform bound~\eqref{ri3} holds true,
we may conclude that
\begin{equation}\label{weak2}
  \frac{\phi_{n_j}}{\sqrt{\Euler^{-s_{n_j}}+\rho^2}}
  \ \xrightarrow[j\to\infty]{w} \
  \frac{\phi_\infty}{\rho}
  \,\quad\mbox{in}\quad
  \mathcal{H}
  \,.
\end{equation}
In particular, recalling~\eqref{implication},
$\phi_\infty \in \Dom(\mathsf{l}_\infty)$.

Now we have all the ingredients needed to pass
to the limit as $n \to \infty$ in~\eqref{re}.
Taking any test function $v \in C_0^\infty\big(S^{d-1}\times(0,\infty)\big)$
in~\eqref{re}, with~$n$ being replaced by~$n_j$,
and sending~$j$ to infinity, we easily obtain from~\eqref{weak1} the identity
\begin{equation}\label{identity}
  \mathsf{l}_\infty(v,\phi_\infty) = (v,f)_\mathcal{H}
  \,.
\end{equation}
Since $C_0^\infty\big(S^{d-1}\times(0,\infty)\big)$
is a core of~$\mathsf{l}_\infty$,
then \eqref{identity} holds true for any $v\in \Dom(\mathsf{l}_\infty)$.
We conclude that $\phi_\infty = \mathsf{L}_\infty^{-1} f$,
for \emph{any} weak limit point of $\{\phi_n\}_{n \in \Nat}$.
From the strong convergence of $\{\phi_{n_j}\}_{j\in\Nat}$,
we eventually conclude with~\eqref{want}.
Recalling Lemma~\ref{Lem.nrs}, the theorem is proved.
\end{proof}
\begin{Remark}
In the two-dimensional situation studied in~\cite{K7},
only the strong-resolvent convergence of the operators was proved.
\end{Remark}
\begin{Remark}
We emphasise that the usage of the magnetic Hardy inequality~\eqref{Hardy.magnet.LW}
of Theorem~\ref{Thm.Hardy.LW} \emph{without} the logarithm
is crucial in the preceding proof.
\end{Remark}

\subsection{Spectral consequences: proof of Theorem~\ref{Thm.main}}
As a consequence of Theorem~\ref{Prop.strong},
the spectrum of~$\mathsf{L}_s$ converges
to the spectrum of~$\mathsf{L}_\infty$ as $s \to \infty$.
As for the latter, we have, in analogy with Proposition~\ref{Prop.Laguerre},
\begin{Proposition}\label{Prop.Laguerre.infinity}
Let $d \geq 2$. Suppose that~$B$ is smooth, closed and compactly supported.
We have
$$
  \sigma(L_\infty) =
  \left\{n+\frac{1+\sqrt{\nu_{B,\ell}(\infty)}}{2}\right\}_{n,\ell \in \Nat}
  \,,
$$
where $\{\nu_{B,\ell}(\infty)\}_{\ell \in \Nat}$ is the set of eigenvalues
of the operator $\big(-i \nabla_\sigma-\mathsf{A}_\infty(\sigma)\big)^2$
in $\sii(S^{d-1})$.
\end{Proposition}

We note that the lowest eigenvalue coincides with $\nu_B(\infty)$.
As a consequence of Theorem~\ref{Prop.strong}
and Proposition~\ref{Prop.Laguerre.infinity}
together with Proposition~\ref{Prop.log} and \eqref{d.high}, we obtain
\begin{Corollary}\label{Corol.strong}
Let $d \geq 2$.
Suppose that~$B$ is smooth, closed and compactly supported.
Then
$$
  \lambda_B(\infty) = \frac{1+\sqrt{\nu_{B}(\infty)}}{2}
  \,.
$$
\end{Corollary}

Recalling Proposition~\ref{Prop.reduction}.(i),
we have thus established the following result.
\begin{Theorem}\label{Thm.main.bis}
Let $d \geq 2$.
Suppose that~$B$ is smooth, closed and compactly supported.
If $\nu_B(\infty) \not= 0$, then there exists a positive constant~$\alpha_{d,B}$
such that
$$
  \big\|\Euler^{-t H_B}\big\|_{\sii_w(\Real^d)\to\sii(\Real^d)}
  \leq
  (1+t)^{-(\alpha_{d,B}+1/2)}
  \,.
$$
\end{Theorem}
\begin{proof}
By virtue of Propositions~\ref{Prop.reduction} and~\ref{Prop.strict},
it is enough to choose
$
  \alpha_{d,B} := \lambda_B^\mathrm{min} - 1/2
$.
\end{proof}

We do not highlight this theorem in the introduction,
because the statement is void for $d \geq 3$
and the two-dimensional situation was already established in~\cite{K7}.

We conclude this section by completing the proof of Theorem~\ref{Thm.main}.
\begin{proof}[Proof of Theorem~\ref{Thm.main}]
By Proposition~\ref{Prop.reduction}.(ii) and Corollary~\ref{Corol.strong},
we have $\gamma_B \geq \lambda_B(\infty)$ and it remains to show that
there is actually an equality. As in Proposition~\ref{Prop.optimal},
it is enough to find an initial datum $u_0 \in \sii_w(\Real^d)$
such that the solution of~\eqref{Cauchy} satisfies the inequality
$\|u(t)\|_{\sii(\Real^d)} \geq c \, (1+t)^{-\lambda_B(\infty)}$
for all $t \geq 0$ with some positive constant~$c$
that may depend on~$u_0$.
Now we choose $u_0(x):=w(x)^{-1/2} \psi_1(x)$, where
$
  \psi_1(y) := |y|^{-(d-2)/2} |y|^{\nu_B(\infty)} \Euler^{-|y|^2/8}
$
is an eigenfunction of~$L_\infty$
corresponding to the eigenvalue~$\lambda_B(\infty)$.
By Proposition~\ref{Prop.justify}, the function~$\tilde{v}$
defined by~\eqref{ISST} and~\eqref{noweight} solves~\eqref{Cauchy.noweight}
(interpreted as~\eqref{weak.form.weight})
with the initial datum $v_0 = \psi_1$.
Let~$R$ be the radius of an open ball~$D_R$ containing the support of~$B$.
Since $\mathsf{A}_s = \mathsf{A}_\infty$ in the exterior
of a shrinking ball~$D_{R_s}$ with $R_s := \Euler^{-s/2} R$,
we have the explicit solution
$$
  \forall s \geq 0 \,, \ y \in \Real^d \setminus D_{R_s}
  \,, \qquad
  \tilde{v}(y,s) = \Euler^{-s \lambda_B(\infty)} \psi_1(y)
  \,.
$$
Recalling~\eqref{preserve} and~\eqref{noweight} again, we get
\begin{multline*}
  \|u(t)\|_{\sii(\Real^d)}
  = \|w^{-1/2}\tilde{v}(s)\|_{\sii(\Real^d)}
  \geq \|w^{-1/2}\tilde{v}(s)\|_{\sii(\Real^d\setminus D_{R_s})}
  \\
  = \Euler^{-s \lambda_B(\infty)}
  \|u_0\|_{\sii(\Real^d\setminus D_{R_s})}
  \geq \Euler^{-s \lambda_B(\infty)}
  \|u_0\|_{\sii(\Real^d\setminus D_{R})}
  =(1+t)^{-\lambda_B(\infty)} \, \|u_0\|_{\sii(\Real^d\setminus D_{R})}
\end{multline*}
for all $t \geq 0$,
where the last identity follows by the relationship~\eqref{relationship}.
\end{proof}
\begin{Remark}
It follows from the precedent proof that there is
an equality in Proposition~\ref{Prop.reduction}.(ii).
\end{Remark}
%

\section{Conclusions}\label{Sec.end}
%
The method of this paper shows that the solutions
of the heat equation~\eqref{Cauchy} for large time behave
as if a local magnetic field~$B$ were replaced by
the singular magnetic field generated by~$\mathsf{A}_\infty$
which reflects the behaviour of the original magnetic field at the space infinity.
We proved that the asymptotic behaviour
of solutions of the heat equation,
when characterised by the polynomial decay rate~$\gamma_B$ of~\eqref{rate},
is determined by the lowest eigenvalue~$\nu_B(\infty)$
of an eigenvalue problem for the magnetic Laplace-Beltrami operator
with the vector potential~$\mathsf{A}_\infty$ on the sphere~$S^{d-1}$.

We also proved new magnetic Hardy inequalities
that played a central role in our proofs.
One approach (though not the most general)
to establish the inequalities
employed the positivity of $r \mapsto \nu_B(r)$
in the spirit of the original work of Laptev and Weidl~\cite{Laptev-Weidl_1999}.
For the appearance of cross-sectional spectral quantities
in a much more general geometric setting,
see~\cite{Mazzeo-MaOwen_2001}.

It is interesting to recall from~\cite{K7} that~$\mathsf{A}_\infty$
corresponds in $d=2$ to the Aharonov-Bohm magnetic field
with the same total flux as~$B$.
Consequently, the presence of magnetic field leads to an improvement
of the large-time decay of solutions of the heat equation~\eqref{Cauchy} in the plane,
provided that the total magnetic flux does not
belong to a discrete set of flux quanta.
Let us also mention in the two-dimensional context
the work of Kova\v{r}\'ik~\cite{Kovarik_2010},
who analysed the large-time behaviour of the heat kernel of~\eqref{semigroup}
in the case of a radially symmetric magnetic field in $d=2$
and obtained sharp two-sided estimates in this special setting.

In higher dimensions, we have $\nu_B(\infty)=0$,
so the transient effect of the magnetic field
is not observable in the present setting through
the polynomial decay rate~$\gamma_B$.
Anyway, because of the presence of magnetic Hardy inequalities,
we expect that there is always an improvement in the decay
of the heat semigroup~\eqref{semigroup} whenever $B \not= 0$.
A possible way how to show it would be to compute
the next term in the asymptotic expansion
of the eigenvalue $\lambda_B(s)$ as $s\to\infty$.

More generally, recall that we expect that there is always
an improvement of the decay rate for the heat semigroup
of an operator satisfying a Hardy-type inequality
(\cf~\cite[Conjecture in Sec.~6]{KZ1} and~\cite[Conjecture~1]{FKP}).
The present paper confirms this general conjecture in the particular case
of two-dimensional magnetic Schr\"odinger operators.

We expect the same decay rates if the assumption about the compact
support of~$B$ is replaced by a fast decay at infinity only.
However, it is quite possible that a slow decay of the field
at infinity will improve the decay of the solutions even further.
In particular, can~$\gamma_B$ be strictly greater than
$\big(1+\sqrt{\nu_B(\infty)}\big)/2$
if~$B$ decays to zero very slowly at infinity?

Throughout this paper, we restricted to smooth~$B$ and~$A$
in order to simplify their mutual relationship
and to state our results in terms of a more physical quantity~$B$.
Omitting the physical interpretation,
it is not difficult to restate our results in terms of~$A$ only,
with a much less restrictive regularity.
For instance, it is clear from the proof of Theorem~\ref{Thm.Hardy}
that it is just enough to assume that~$A$ is bounded
and not exact to get the Hardy inequality~\eqref{Hardy.magnet}.

It would be also interesting to examine the effect
of the presence of the local magnetic field in other physical models.
As one possible direction of this research,
let us mention the question of dispersive estimates
for the Schr\"odinger group generated by~\eqref{Schrodinger}
(see \cite{D'Ancona-Fanelli_2008}, \cite{D'Ancona-Fanelli-Vega-Visciglia_2010},
\cite{Fanelli-Felli-Fontelos-Primo_2013}, \cite{Grillo-Kovarik_2014},
\cite{Fanelli-Felli-Fontelos-Primo_2014}
for a recent study of related problems).

 It is well known that the large-time behaviour of a heat semigroup
is related to low-energy properties of the resolvent of its generator.
In this paper, we studied this relationship
by the method of self-similar variables applied to the heat equation
and using magnetic Hardy-type inequalities
as a characterisation of spectral-threshold properties of the generator.
An alternative approach would be to establish a Laurent expansion
of the resolvent at zero energy and apply it to the heat semigroup
with help of functional calculus.
This approach has been recently undertaken by Kova\v{r}\'ik~\cite{Kovarik_2015}
to study a large-time behaviour of the Schr\"odinger equation
in the present magnetic setting for $d=2$
(the established Laurent expansion can be used for the heat semigroup as well).
More generally, the low-energy properties of the resolvent
are subject of an intensive study in geometric scattering theory;
see \cite{Wang_2006}, \cite{Guillarmou-Hassell_2008}, \cite{Guillarmou-Hassell_2009}
and \cite{Guillarmou-Hassell-Sikora_2013}
for recent developments in a much more general geometric setting.
An extension of these results to the present magnetic case seems
to constitute an interesting open problem.

\appendix
\section{An abstract criterion for the norm-resolvent convergence}\label{Sec.App}
%
In this appendix we establish the following abstract result.

\begin{Lemma}\label{Lem.nrs}
Let $\{R_s\}_{s \in \Real}$ be a family of bounded operators
on a Hilbert space~$\mathfrak{H}$
and let~$R$ be a compact operator in~$\mathfrak{H}$.
Suppose that
\begin{equation}\label{uniform}
  \forall \{s_n\}_{n\in\Nat} \subset \Real \,, \
  \forall \{f_n\}_{n \in \Nat} \subset \mathfrak{H} \,,
  \quad
\left.
\begin{aligned}
  \bullet \ & s_n \xrightarrow[n \to \infty]{} \infty
  \\
  \bullet \ & f_n \xrightarrow[n \to \infty]{w} f
  \ \mbox{in} \ \mathfrak{H}
  \\
  \bullet \ & \forall n \in \Nat \,, \quad
  \|f_n\|_\mathfrak{H} = 1
\end{aligned}
\right\}
\quad\Longrightarrow\quad
  R_{s_n} f_n \xrightarrow[n \to \infty]{} R f
  \ \mbox{in} \ \mathfrak{H}
  \,.
\end{equation}
Then $\{R_s\}_{s \geq 0}$ converges to~$R$ uniformly, \ie,
\begin{equation}\label{nrs.abstract}
  \lim_{s \to \infty}
  \|R_s-R\|_{\mathfrak{H}\to\mathfrak{H}}
  = 0
\end{equation}
\end{Lemma}
\begin{proof}
We prove the lemma by contradiction.
Let us assume that~\eqref{nrs.abstract} is violated, while~\eqref{uniform} holds.
The former implies that there exist a positive number~$c$,
a sequence of numbers
$\{s_n\}_{n \in \Nat} \subset \Real$
such that $s_n \to \infty$ as $n \to \infty$
and a sequence of functions
$\{f_n\}_{n \in \Nat} \subset \mathfrak{H}$ such that
$$
  \|f_{n}\|_{\mathfrak{H}}=1
  \qquad \mbox{and} \qquad
  \big\| R_{s_n} f_{n}-R f_{n} \big\|_{\mathfrak{H}} \geq c > 0
$$
for every $n \in \Nat$.
Since the sequence $\{f_{n}\}_{n\in\Nat}$ is bounded,
it is precompact in the weak topology of~$\mathfrak{H}$.
Let~$f$ be a weak limit point, \ie,
for a sequence $\{n_j\}_{j \in \Nat} \subset \Nat$
such that $n_j \to \infty$ as $j \to \infty$,
$\{f_{n_j}\}_{j \in \Nat}$ converges weakly to~$f$ in~$\mathfrak{H}$.
Since~$R$ is compact,
the sequence of functions $\{R f_{n_j}\}_{j \in \Nat}$
converges strongly to~$R f$ in~$\mathfrak{H}$.
Consequently,
$$
  0 < c \leq
  \lim_{j \to \infty} \big\| R_{s_{n_j}} f_{n_j}-R f_{n_j} \big\|_{\mathfrak{H}}
  = \lim_{j \to \infty} \big\| R_{s_{n_j}} f_{n_j}-R f \big\|_{\mathfrak{H}}
  = 0
  \,,
$$
where the last equality follows from~\eqref{uniform}
and gives a contradiction with the positivity of~$c$.
\end{proof}
%

\subsection*{Acknowledgment}
%
This work was initiated when C.C.\ visited
the Department of Theoretical Physics at the Nuclear Physics Institute
of the Academy of Sciences of the Czech Republic in Re\v{z}.
C.C.\ wishes to acknowledge the hospitality
and the adequate research environment he received there.
The work was partially supported
by the project RVO61389005 and the GACR grant No.\ 14-06818S.
C.C.\ was additionally supported
by the project PN-II-ID-PCE-2011-3-0075 of  the Ministry of National Education,
CNCS-UEFISCDI Romania.
D.K.~also acknowledges the award from
the \emph{Neuron fund for support of science},
Czech Republic, May 2014.

%
\addcontentsline{toc}{section}{References}
\bibliography{bib}
\bibliographystyle{acm}
%
\end{document}